\documentclass[11pt]{article}

\usepackage[T1]{fontenc}
\usepackage[utf8]{inputenc}

\usepackage{lmodern, microtype}
\usepackage{amssymb, latexsym, textcomp}
\usepackage{amsmath, amsthm}
\usepackage{array, setspace, verbatim}

\usepackage[english]{babel}
\usepackage[english]{varioref}

\usepackage{ifthen}
\usepackage{fancyhdr, fancybox, float}
\usepackage{color, multicol}
\usepackage{graphicx, pst-all, pb-diagram}
\usepackage[pdftex]{hyperref}

\newcommand{\espf}{\vspace*{1ex}}
\newcommand{\dst}{\displaystyle}
\newcommand{\nidt}{\noindent}

\newcommand{\ov}[1]{\overline{#1}}
\newcommand{\wt}[1]{\widetilde{#1}}

\newcommand{\del}{\nabla}
\renewcommand{\epsilon}{\varepsilon}
\newcommand{\ii}{\infty}
\renewcommand{\phi}{\varphi}

\newcommand{\C}{\mathbb{C}}
\newcommand{\Cc}{\mathcal{C}}
\newcommand{\D}{\mathbb{D}}
\newcommand{\Dd}{\mathcal{D}}
\newcommand{\Ee}{\mathcal{E}}
\newcommand{\Gg}{\mathcal{G}}
\renewcommand{\H}{\mathbb{H}}
\renewcommand{\L}{\mathbb{L}}
\newcommand{\N}{\mathbb{N}}
\newcommand{\R}{\mathbb{R}}
\renewcommand{\S}{\mathbb{S}}

\newcommand{\Uu}{\mathcal{U}}
\newcommand{\Vv}{\mathcal{V}}

\newcommand{\lp}{\left(}
\newcommand{\rp}{\right)}
\newcommand{\lc}{\left[}
\newcommand{\rc}{\right]}
\newcommand{\lac}{\left\{}
\newcommand{\rac}{\right\}}
\newcommand{\lan}{\langle}
\newcommand{\ran}{\rangle}
\newcommand{\lb}{\left|}
\newcommand{\rb}{\right|}
\newcommand{\lr}{\left.}
\newcommand{\rr}{\right.}

\renewcommand{\a}{\forall}
\newcommand{\eq}{\ \Longleftrightarrow\ }
\newcommand{\into}{\rightarrow}

\newcommand*{\st}[1][\mathrel{|}]{%
  \ifthenelse{\equal{#1}{big}}{\mathrel{\big{|}}}{%
    \ifthenelse{\equal{#1}{Big}}{\mathrel{\st[Big]}}{%
      \ifthenelse{\equal{#1}{bigg}}{\mathrel{\bigg{|}}}{%
        \ifthenelse{\equal{#1}{Bigg}}{\mathrel{\Bigg{|}}}
      }%
    }%
  }%
}

\DeclareMathOperator{\argcosh}{argcosh}
\DeclareMathOperator{\argtanh}{argtanh}
\DeclareMathOperator{\Div}{div}

\newcommand{\der}[2][]{\frac{\partial#1}{\partial#2}}
\newcommand{\Lim}[2]{\lim_{#1\into#2}}
\renewcommand{\matrix}[2]{\lp\begin{array}{@{}*{#1}{c}@{}}#2\end{array}\rp}
\newcommand{\tend}[2]{\xrightarrow[#1\into#2]{}}

\newcommand{\ligne}[2]{\qbezier(#1)(#1)(#2)}

\theoremstyle{definition}
\newtheorem{definition}{Definition}[section]
\newtheorem{remark}[definition]{Remark}

\theoremstyle{plain}
\newtheorem{lemma}[definition]{Lemma}
\newtheorem{proposition}[definition]{Proposition}
\newtheorem{theorem}[definition]{Theorem}
\newtheorem*{theorem*}{Theorem}
\newtheorem{corollary}[definition]{Corollary}

\renewenvironment{proof}[1][]{\nidt\textit{Proof}#1.\ }{\hfill$\square$\bigskip}

\hypersetup{%
pdfauthor= {S{\'e}bastien Cartier and Laurent Hauswirth},%
pdftitle= {Deformations of constant mean curvature 1/2 surfaces in H2xR with vertical ends at infinity},%
pdfcreator= {PDFLaTeX},%
pdfproducer= {PDFLaTeX}%
}

\title{Deformations of constant mean curvature $1/ 2$ surfaces in $\H^2 \times \R$ with vertical ends at infinity}
\author{S{\'e}bastien Cartier and Laurent Hauswirth}

\begin{document}

\maketitle

\begin{abstract}

\nidt \itshape We study constant mean curvature $1/ 2$ surfaces in $\H^2 \times \R$ that admit a compactification of the mean curvature operator. We show that a particular family of complete entire graphs over $\H^2$ admits a structure of infinite dimensional manifold with local control on the behaviors at infinity. These graphs also appear to have a half-space property and we deduce a uniqueness result at infinity. Deforming non degenerate constant mean curvature $1/ 2$ annuli, we provide a large class of (non rotational) examples and construct (possibly embedded) annuli without axis, i.e. with two vertical, asymptotically rotational, non aligned ends.

\end{abstract}

\nidt \textit{Mathematics Subject Classification:} \emph{53A10, 53C42}.

\section{Introduction}

This paper concerns the theory of constant mean curvature (\emph{CMC} for short) surfaces $H= 1/ 2$ in $\H^2 \times \R$. The value $H= 1/ 2$ is critical in the sense that there is no compact CMC surface for $H \leq 1/ 2$ while for $H> 1/ 2$ there are rotational compact examples. A half-space theorem in $\H^2 \times \R$ (see \cite{HaRoSp}) proves that for CMC $H= 1/ 2$, complete multigraphs are entire graphs over $\H^2$. Entire graphs are classified by I.~Fern{\'a}ndez and P.~Mira \cite{FeMi2} and their moduli space is modeled on the set of quadratic holomorphic differential $Q$ defined on the complex plane $\C$ or the unit disk $\D$. The link between $Q$ and the geometry of the graph is not very well understood.

\medskip

We first deal with complete conformal immersions of the disk $\D$, properly immersed into the half-space $\H^2 \times \R_+$ ($x_3 \geq 0$), which are entire vertical graphs over $\H^2$. We assume that the third coordinate $x_3 \into +\ii$ on any diverging sequence of points in $\D$, which means the height function is proper. Up to this date, the only simply connected example is a rotational example called the \emph{hyperboloid} $S_0$. In the Poincar{\'e} disk model of $\H^2 \times \R$ ---~see~\eqref{eq:modpoinc} below~--- with polar coordinates $(r, \theta)$, a parametrization of $S_0$ as a graph over $\H^2$ is:
\[
(r, \theta) \in [0, 1) \times \S^1 \mapsto \lp re^{i\theta}, \frac{2}{\sqrt{1- r^2}} \rp \in \H^2 \times \R.
\]

We describe a family of examples endowed with a structure of infinite dimensional smooth manifold. The manifold structure arises from a suitable \emph{compactification} of the mean curvature operator at infinity (Theorem~\ref{thm:compcourb}) and is diffeomorphic to a codimension one submanifold of $\Cc^{2, \alpha}(\S^1) \times \R$ (Theorem~\ref{thm:structg}). This construction comes with a control of the asymptotic behavior in terms of the horizontal (hyperbolic) distance from the hyperboloid $S_0$, namely:

\begin{theorem*}[Theorem~\ref{thm:treibergsloc}]

For any small $\gamma \in \Cc^{2, \alpha}(\S^1)$ such that $e^{-\gamma}$ has unit $L^2(\S^1)$-norm, there exists a CMC-$1/ 2$ complete entire graph at asymptotic horizontal signed distance $2\gamma$ from $S_0$.

\end{theorem*}

These graphs are interesting, since any connected complete embedded CMC-$1/ 2$ surface in $\H^2 \times \R$ which is contained in the half-space $\H^2 \times \R_+$ and has a proper height function is a vertical entire graph. Indeed, apply Alexandrov reflection principle to such an immersion with respect to the horizontal slices. Namely, reflect through the slice the part of the surface situated below it to obtain a surface which is a bigraph i.e. a graph over each side of the slice. There will be no first point of tangent contact between the initial surface and the part of the bigraph which is not a part of the surface, since there is no compact CMC-$1/ 2$ surface in $\H^2 \times \R$.

We also prove a half-space property for these entire graphs:

\begin{theorem*}[Theorem~\ref{thm:halfspace}]

Let $\Sigma$ be a CMC-$1/ 2$ surface which is properly immersed in $\H^2 \times \R$, lies on one side of a CMC-$1/ 2$ entire graph $S$ in the aforementioned family and is well oriented with respect to $S$. Then $\Sigma$ coincides with $S$ up to a vertical translation.

\end{theorem*}

\nidt The ``well oriented'' assumption is in the sense of L.~Mazet \cite{Ma} and means that if $\Sigma$ is below $S$, its mean curvature vector points to $S$. We use this result to show an asymptotic rigidity in our family of CMC-$1/ 2$ entire graphs (Theorem~\ref{thm:rigidinf}). Namely, if two graphs in the family are at the same asymptotic horizontal signed distance from the hyperboloid $S_0$, they coincide up to a vertical translation.

\medskip

In $\H^2 \times \R$, R.~S{\'a}~Earp and E.~Toubiana \cite{SETo} construct a one-parameter family of CMC $H= 1/ 2$ annuli which are rotationally invariant around a vertical geodesic. Recently, L.~Mazet has shown \cite{Ma2} that for $H> 1/ 2$, CMC annuli which are cylindrically bounded around a vertical geodesic are rotational examples.

Though annuli are not cylindrically bounded for $H= 1/ 2$, we prove that in a bounded tubular neighborhood of a rotational example, there are annuli, eventually embedded, which are asymptotic to different rotational examples with different axis:

\begin{theorem*}[Theorem~\ref{thm:nonalign}]

There exist CMC-$1/ 2$ annuli in $\H^2 \times \R$ with vertical ends, that are asymptotic ---~regarding the horizontal hyperbolic distance~--- to rotational examples with different vertical axis.

\end{theorem*}

It means that contrary to the case of embedded minimal surfaces in $\R^3$ with finite total curvature and horizontal ends \cite{PeRo}, the notion of \emph{axis} is not relevant in general for CMC-$1/ 2$ annuli with vertical ends in $\H^2 \times \R$.

\subsection*{Notations}

Let $\D= \lac z \in \C \st[big] |z|< 1 \rac$ be the open unit disk, $\ov{\D}= \lac z \in \C \st[big] |z| \leq 1 \rac$ its closure and $(r, \theta)$ the polar coordinates on $\ov{\D}$. We use two standard models of $\H^2 \times \R$, which are the Minkowski model:
\begin{multline} \label{eq:modmink}
\H^2 \times \R= \Big{(} \lac \lr (x_0, \dots, x_3) \in \R^4 \rb x_1^2+ x_2^2- x_0^2= -1 \rac, \\
ds_L^2= dx_1^2+ dx_2^2+ dx_3^2- dx_0^2 \Big{)},
\end{multline}
where $\H^2 \times \R$ is seen as a subspace of the $4$-dimensional Minkowski space $\L^4$, and the Poincar{\'e} disk model:
\begin{multline} \label{eq:modpoinc}
\H^2 \times \R= \bigg{(} \lac (w, x_3) \in \D \times \R \rac, \\
ds_P^2+ dx_3^2= \frac{4}{(1- |w|^2)^2} |dw|^2+ dx_3^2 \bigg{)}.
\end{multline}
The vector field associated to the third coordinate is denoted $e_3$. In the Poincar{\'e} disk model~\eqref{eq:modpoinc}, the hyperbolic radius $\rho_{\H}(w)$ of a point $w$ is:
\[
\rho_{\H}(w)= 2\argtanh |w|= \log \lp \frac{1+ |w|}{1- |w|} \rp,
\]
and we will need the following formula in the proof of Proposition~\ref{prop:graphdeform}:
\[
\cosh \frac{\rho_{\H}(w)}{2}= \frac{1}{\sqrt{1- |w|^2}}.
\]
We call vertical graphs in $\H^2 \times \R$, immersions which are complete graphs over an open subset of the slice $\H^2 \equiv \H^2 \times \lac 0 \rac$, and we call vertical annuli in $\H^2 \times \R$, immersions which are complete vertical bigraphs.

Given surfaces $S, S'$ in $\H^2 \times \R$ admitting parametrizations in the Poincar{\'e} disk model respectively:
\[
\big{(} f(t, \theta)e^{i\theta}, t \big{)} \quad \text{and} \quad \big{(} f'(t, \theta)e^{i\theta}, t \big{)},
\]
the hyperbolic horizontal signed distance $d_{\H}(S, S')(t, \theta)$ between $S$ and $S'$ at height $t$ and in the direction $\theta$ is the difference of their hyperbolic radii in the slice $\H^2 \times \lac t \rac$ and direction $\theta$:
\begin{align*}
d_{\H}(S, S')(t, \theta) & = \rho_{\H}(S')(t, \theta)- \rho_{\H}(S)(t, \theta) \\
                         & = 2\lp \argtanh f'(t, \theta)- \argtanh f(t, \theta) \rp.
\end{align*}
When it exists, the \emph{asymptotic hyperbolic horizontal signed distance} between $S$ and $S'$ in the direction $\theta$ is the limit $\dst \Lim{t}{+\ii} d_{\H}(S, S')(t, \theta)$.

For any $R \in [0, 1)$, let $\Omega_R \subset \D$ be the domain $\Omega_R= \lac R \leq r< 1 \rac$. We consider the set of \emph{admissible domains} $\Dd= \lac \Omega_R | 0 \leq R< 1 \rac$. The boundary at infinity $\partial_{\ii} \H^2$ of $\H^2$ is identified with $\S^1$.

Given $\Omega \in \Dd$, the spaces $\Cc^{k, \alpha}(\ov{\Omega})$ and $\Cc^{k, \alpha}_0(\ov{\Omega})$, with $k \geq 0$ and $0< \alpha< 1$, are respectively the usual Hölder space and the subspace of functions that are zero on the boundary of $\Omega$. Finally, we consider the spaces $L^2(\cdot)$ endowed with the natural scalar product denoted $\lan \cdot, \cdot \ran_{L^2(\cdot)}$ and Hilbert norm $|\cdot|_{L^2(\cdot)}$.

\section{The mean curvature operator} \label{sec:bouts}

Consider a surface $S$ parametrized by an immersion $X: \D \into \H^2 \times \R$ with complete induced metric $g$. By \emph{compactification} of $S$, we mean a conformal change $\ov{g}$ of metric such that $\ov{g}$ extends to a metric on $\ov{\D}$.

The process is sensible to the parametrization. For instance, consider the hyperboloid $S_0$. It is a vertical graph over $\H^2$ parametrized by:
\[
(r, \theta) \in \D \mapsto \lp re^{i\theta}, \frac{2}{\sqrt{1- r^2}} \rp \in \H^2 \times \R,
\]
in the Poincar{\'e} disk model~\eqref{eq:modpoinc}, with induced metric:
\[
g= \frac{4}{(1- r^2)^3}\matrix{2}{2- r^2 & 0 \\ 0 & 1- r^2}.
\]
But $g$ cannot be conformally extended to the boundary $\lac r= 1 \rac$ of $\D$, since the terms of $g$ have different rates of explosion when $r \into 1$. The resulting metric would degenerate for $r= 1$.

To ensure the extension of the induced metric, we use a conformal parametrization $S_0$, namely the immersion $X^0: \D \into \H^2 \times \R$ defined by:
\[
X^0(r, \theta)= \lp F(r, \theta), \frac{2}{\sqrt{1- |F(r, \theta)|^2}} \rp= \lp F(r, \theta), 2\frac{1+ r^2}{1- r^2} \rp,
\]
where $F: \D \into \H^2$ is the $\Cc^1$-diffeomorphism defined in the Poincar{\'e} disk model~\eqref{eq:modpoinc} by:
\[
F(r, \theta)= \frac{2r}{1+ r^2} e^{i\theta}
\]
and in the Minkowski model~\eqref{eq:modmink} by:
\begin{gather*}
F(r, \theta)= \big{(} \cosh \chi(r, \theta), \sinh \chi(r, \theta) \cos \theta, \sinh \chi(r, \theta) \sin \theta \big{)} \\
\text{with} \quad \chi(r, \theta)= 2\log \lp \frac{1+ r}{1- r} \rp.
\end{gather*}

\begin{definition}

A surface in $\H^2 \times \R$ is said to \emph{admit graph coordinates at infinity}, if there exist an admissible domain $\Omega \in \Dd$ and a function $h: \Omega \into \R$ such that a part of the surface can be parametrized as the immersion on $\Omega$:
\[
X: (r, \theta) \in \Omega \mapsto \big{(} F(r, \theta), h(r, \theta) \big{)} \in \H^2 \times \R.
\]
When defined, we call such a parametrization \emph{graph coordinates at infinity}.

\end{definition}

In the sequel, we use graph coordinates at infinity to compactify surfaces and quantify their asymptotic behavior. Surfaces are thus considered as \emph{compact surfaces with boundary} and we can apply the method first developed by B.~White in \cite{Wh}.

\subsection[The family E]{The family $\Ee$}

Let $\Ee$ be the set of immersed surfaces in $\H^2 \times \R$, which admit ---~up to a symmetry with respect to the slice $\H^2 \times \lac 0 \rac$~--- graph coordinates at infinity written as:
\begin{equation} \label{eq:graphcoord}
X^{\eta}: (r, \theta) \in \Omega \mapsto \lp F(r, \theta), 2e^{\eta(r, \theta)} \frac{1+ r^2}{1- r^2} \rp \in \H^2 \times \R,
\end{equation}
for some admissible domain $\Omega \in \Dd$ and $\eta \in \Cc^{2, \alpha}(\ov{\Omega})$. Elements of $\Ee$ have \emph{vertical ends} \cite{ElNeSE} i.e. topological annuli with no asymptotic point at finite height ---~i.e. topological annuli properly embedded in $(\H^2 \cup \partial_{\ii} \H^2) \times \R$.

The hyperboloid $S_0$ itself is in $\Ee$ with $\Omega= \D$ and $\eta \equiv 0$. And so are the rotational examples of E.~Toubiana and R.~S{\'a}~Earp studied in Section~\ref{sec:ann}, owing the asymptotic development~\eqref{eq:asdev}.

\medskip

We highlight two properties of the family $\Ee$. The first is that it contains normal deformations of the hyperboloid $S_0$. Namely:

\begin{proposition} \label{prop:graphdeform}

A normal graph $S= \exp_{S_0}(\zeta N)$ over $S_0$, where $N$ is the upward pointing normal to $S_0$ and $\zeta \in \Cc^{2, \alpha}(\ov{\D})$, is in $\Ee$. In other words, there exist $\Omega \in \Dd$ and $\eta \in \Cc^{2, \alpha}(\ov{\Omega})$ such that the end of $S$ admits graph coordinates at infinity as in~\eqref{eq:graphcoord}. \\
Furthermore, the asymptotic value of $\eta$ is linked with the asymptotic horizontal (hyperbolic) distance between $S$ and $S_0$:
\[
\eta|_{\partial \D}= \frac{1}{2} \zeta|_{\partial \D},
\]

\end{proposition}

\begin{proof}
We use the Minkowski model~\eqref{eq:modmink} of $\H^2 \times \R$, where the map $F$ reads:
\begin{gather*}
F(r, \theta)= \big{(} \cosh \chi(r, \theta), \sinh \chi(r, \theta) \cos \theta, \sinh \chi(r, \theta) \sin \theta \big{)} \\
\text{with} \quad \chi(r, \theta)= 2\log \lp \frac{1+ r}{1- r} \rp.
\end{gather*}
A computation shows the unit normal $N$ to $S_0$ is:
\begin{multline*}
N= -\frac{2r}{1+ r^2} \lp \sinh \chi \der{x_0}+ \cosh \chi \cos \theta \der{x_1}+ \cosh \chi \sin \theta \der{x_2} \rp \\
+\frac{1- r^2}{1+ r^2} \der{x_3},
\end{multline*}
in the canonical basis of $\L^4$. Hence, $S$ is parametrized by the immersion:
\begin{multline*}
\Bigg{(} \cosh \lp \chi- \frac{2r\zeta}{1+ r^2} \rp, \sinh \lp \chi- \frac{2r\zeta}{1+ r^2} \rp \cos \theta, \\
\sinh \lp \chi- \frac{2r\zeta}{1+ r^2} \rp \sin \theta, 2\frac{1+ r^2}{1- r^2}+ \frac{1- r^2}{1+ r^2} \zeta \Bigg{)}.
\end{multline*}
We want to find new coordinates $(\wt{r}, \wt{\theta})$ on an admissible domain verifying:
\[
\chi(\wt{r}, \wt{\theta})= \chi(r, \theta)- \frac{2r}{1+ r^2} \zeta(r, \theta), \quad \cos \wt{\theta}= \cos \theta \quad \text{and} \quad \sin \wt{\theta}= \sin \theta,
\]
to have graph coordinates at infinity on $S$ as in~\eqref{eq:graphcoord}. Taking $\wt{\theta}= \theta$, compute:
\begin{align*}
\der{r} \lp \chi(r, \theta)- \frac{2r}{1+ r^2} \zeta(r, \theta) \rp & = \frac{4}{1- r^2}- \frac{2}{1+ r^2} \lp \frac{1- r^2}{1+ r^2} \zeta+ r\zeta_r \rp \\
                                                                    & = \frac{4}{1- r^2}+ O(1).
\end{align*}
If $r$ is sufficiently close to $1$, the map $r \mapsto \chi- 2r\zeta/ (1+ r^2)$ is strictly increasing (uniformly in $\theta$), which ensures existence and uniqueness of $\wt{r}$.

\medskip

To compute the asymptotic horizontal distance, consider a horizontal slice $\H^2 \times \lac t \rac$ intersecting $S$ and $S_0$. The hyperbolic radii of $S$ and $S_0$ at height $t$ and in the direction $\theta$ respectively denoted $\rho_{\H}(S)(t, \theta)$ and $\rho_{\H}(S_0)(t, \theta)$ verify:
\begin{gather*}
t= \frac{2e^{\eta}}{\sqrt{1- |F|^2}}= 2e^{\eta} \cosh \frac{\rho_{\H}(S)(t, \theta)}{2} \\
\text{and} \quad t= \frac{2}{\sqrt{1- |F|^2}}= 2\cosh \frac{\rho_{\H}(S_0)(t, \theta)}{2},
\end{gather*}
and we deduce:
\begin{gather*}
\rho_{\H}(S)(t, \theta)= 2\argcosh \frac{te^{-\eta}}{2}= 2\log t- 2\eta+ O\lp \frac{1}{t^2} \rp \\
\text{and} \quad \rho_{\H}(S_0)(t, \theta)= 2\argcosh \frac{t}{2}= 2\log t+ O\lp \frac{1}{t^2} \rp.
\end{gather*}
Therefore, the hyperbolic horizontal signed distance $d_{\H}(S, S_0)(t, \theta)$ between $S$ and $S_0$ at height $t$ and in the direction $\theta$ is:
\[
d_{\H}(S, S_0)(t, \theta)= \rho_{\H}(S_0)(t, \theta)- \rho_{\H}(S)(t, \theta)= 2\eta+ O\lp \frac{1}{t^2} \rp,
\]
which establishes the equality $\zeta|_{\partial \D}= 2\eta|_{\partial \D}$ at infinity. Indeed, $\zeta|_{\partial \D}$ is the normal signed distance between $S$ and $S_0$ at infinity as $S$ is constructed as a normal graph over $S_0$ at signed distance $\zeta$. And $\zeta|_{\partial \D}$ is also the horizontal distance at infinity, since the normal $N$ is asymptotically horizontal, which means $\lan N, e_3 \ran \longrightarrow 0$.
\end{proof}

Proposition~\ref{prop:graphdeform} emphasizes the fact that the relevant information at infinity is the asymptotic horizontal distance from the hyperboloid. And as suggested by~\eqref{eq:asdev} in Section~\ref{sec:ann}, the asymptotic horizontal distance is also relevant for deformed annuli, since the rotational examples are at a finite constant asymptotic horizontal distance from each other.

Therefore a general principle in this paper is to fix a convenient surface, the \emph{model surface}, and to construct deformations of the model surface prescribing the asymptotic horizontal distance from the model surface. It is also the supporting idea of the compactification of the mean curvature operator (Theorem~\ref{thm:compcourb}).

\medskip

A second interesting property of $\Ee$ is the following:

\begin{proposition} \label{prop:invisom}

The image of any element of $\Ee$ under the action of any isometry of $\H^2 \times \R$ is still an element of $\Ee$.

\end{proposition}

\begin{proof}
Consider a surface $S \in \Ee$ with graph coordinates at infinity $(F, h)$ defined on $\Omega \in \Dd$, and denote by $(F, h')$ the graph coordinates at infinity of its image $S'$ under an isometry $\psi$ of $\H^2 \times \R$. Using parametrization~\eqref{eq:graphcoord}, we know that in the Poincar{\'e} disk model~\eqref{eq:modpoinc}:
\[
h= \frac{2e^{\eta}}{\sqrt{1- |F|^2}} \quad \text{with} \quad \eta \in \Cc^{2, \alpha}(\ov{\Omega}).
\]
It is sufficient to examine the cases when $\psi$ is either an isometry of $\H^2$ fixing the coordinate $x_3$ or a vertical translation. If $\psi$ is a vertical translation of $t_0 \in \R$, we have:
\[
h'= \frac{2e^{\eta}}{\sqrt{1- |F|^2}}+ t_0= 2\exp \lp \eta+ \log \lp 1+ t_0 \frac{e^{-\eta}}{2} \frac{1- r^2}{1+ r^2} \rp \rp \frac{1}{\sqrt{1- |F|^2}},
\]
eventually after a restriction to a domain $\Omega' \in \Dd$ for which $h|_{\Omega'}> -t_0$.

If $\psi$ reduces to an isometry of $\H^2$ preserving the orientation of $\H^2$, there exist $w_0 \in \D$ and $\delta_0 \in \R$ such that:
\[
\psi(w)= \frac{w+ w_0}{1+ \ov{w_0} w} e^{i\delta_0}.
\]
If $\psi'= F^{-1} \circ \psi^{-1} \circ F$, then:
\begin{align*}
h' & = h \circ \psi'= \frac{2e^{\eta \circ \psi'}}{\sqrt{1- |\psi^{-1} \circ F|^2}}= \lp e^{\eta \circ \psi'} \frac{\lb 1- \ov{w_0} F \rb}{\sqrt{1- |w_0|^2}} \rp \frac{2}{\sqrt{1- |F|^2}} \\
   & = \exp \lp \eta \circ \psi'+ \log \lp \frac{\lb 1- \ov{w_0} F \rb}{\sqrt{1- |w_0|^2}} \rp \rp \frac{2}{\sqrt{1- |F|^2}},
\end{align*}
and $S' \in \Ee$. Changing $F$ in $\ov{F}$, gives the result when $\psi$ reduces to an isometry of $\H^2$ reversing the orientation.
\end{proof}

\begin{remark}

The value $\eta|_{\partial \D}$ is invariant under vertical translations.

\end{remark}

\subsection{Compactification of the mean curvature}

From now on, to ease the notations, we denote with indices $1, 2$ quantities related to coordinates $r, \theta$ respectively. Consider an admissible domain $\Omega \in \Dd$ and a function $a \in \Cc^{2, \alpha}(\ov{\Omega})$. The model surface is the immersion $X^a$, written as in~\eqref{eq:graphcoord}, and we are interested in deformations $X^{\eta}$ with $\eta= a+ \xi$.

\begin{theorem} \label{thm:compcourb}

For any deformation $X^{a+ \xi}$ of the model surface $X^a$, with $\xi \in \Cc^{2, \alpha}(\ov{\Omega})$, the respective mean curvatures $H(a+ \xi)$ and $H(a)$ verify the following:
\begin{equation} \label{eq:comph}
\sqrt{|g(a)|} \big{(} H(a+ \xi)- H(a) \big{)}= \sum_{i, j} A_{ij}(r, \theta, a, D\xi) \xi_{ij}+ B(r, \theta, a, \xi, D\xi),
\end{equation}
where $|g(a)|$ is the determinant of the metric induced by $X^a$, $A_{ij}$ and $B$ are $\Cc^{0, \alpha}$ functions on $\ov{\Omega}$ which are real-analytic in their variables, and $A= (A_{ij})$ is a coercive matrix on $\ov{\Omega}$.

\end{theorem}

\begin{proof}[ (See Appendix~\ref{sec:comp} for computation details)]
Denote $\sigma$ the pullback metric $F^* ds_P^2$, i.e. in matrix terms:
\[
\sigma= \frac{16}{(1- r^2)^4} \matrix{2}{(1- r^2)^2 & 0 \\ 0 & r^2(1+ r^2)}.
\]
Differential properties of a surface in $\H^2 \times \R$ with graph coordinates at infinity $(F, h)$ are the ones of the actual graph of $h$ in $\D \times \R$ endowed with the metric $\sigma+ dx_3^2$. Following J. Spruck \cite{Sp}, the mean curvature $H(a+ \xi)$ is:
\[
H(a+ \xi)= \frac{1}{2} \Div_{\sigma} \lp \frac{\del_{\sigma} h(a+ \xi)}{W(a+ \xi)} \rp \quad \text{with} \quad W(a+ \xi)= \sqrt{1+ |\del_{\sigma} h(a+ \xi)|_{\sigma}^2},
\]
with quantities computed with respect to $\sigma$. If $(\Gamma_{ij}^k)$ denote the Christoffel symbols associated to $\sigma$, we have:
\[
H(a+ \xi)= \frac{1}{2W(a+ \xi)} \sum_{i, j} g^{ij}(a+ \xi) \lp \partial_{ij} h(a+ \xi)- \sum_k \Gamma_{ij}^k \partial_k h(a+ \xi) \rp,
\]
where the non zero Christoffel symbols are:
\begin{gather*}
\Gamma_{11}^1= \frac{2r}{1- r^2}, \quad \Gamma_{12}^2= \Gamma_{21}^2= \frac{1+ 6r^2+ r^4}{r(1+ r^2)(1- r^2)} \\
\text{and} \quad \Gamma_{22}^1= -\frac{r(1+ r^2)(1+ 6r^2+ r^4)}{(1- r^2)^3}.
\end{gather*}
If $a_1$ (resp. $a_2$) denotes the derivative of $a$ with respect to $r$ (resp. $\theta$), the induced metric $g(a)$ reads:
\begin{gather*}
g_{11}(a)= \frac{16(1+ r^2)^2 e^{2a}}{(1- r^2)^4} \Bigg{[} 1+ \frac{2ra_1}{1+ r^2} (1- r^2)+ \lp \frac{a_1^2}{4}+ \frac{e^{-2a}- 1}{(1+ r^2)^2} \rp (1- r^2)^2 \Bigg{]}, \\
g_{12}(a)= \frac{8(1+ r^2)^2 a_2 e^{2a}}{(1- r^2)^3} \lc \frac{2r}{1+ r^2}+ \frac{a_1}{2} (1- r^2) \rc \\
\text{and} \quad g_{22}(a)= \frac{16r^2 (1+ r^2)^2}{(1- r^2)^4} \lc 1+ \frac{a_2^2 e^{2a}}{4r^2} (1- r^2)^2 \rc,
\end{gather*}
and the expression of $W(a)$ is the following:
\begin{multline} \label{eq:exprw}
W(a)= \frac{(1+ r^2) e^a}{1- r^2} \Bigg{[} 1+ \frac{2ra_1}{1+ r^2} (1- r^2)+ \lp \frac{a_1^2}{4}+ \frac{e^{-2a}- 1}{(1+ r^2)^2} \rp (1- r^2)^2 \\
+\frac{a_2^2}{4r^2 (1+ r^2)^2} (1- r^2)^4 \Bigg{]}^{1/ 2}.
\end{multline}
The computation detailed in Appendix~\ref{sec:comp} gives the expression~\eqref{eq:comph} with the desired regularity and:
\[
A_{11}= e^{-a}+ O(1- r^2), \quad A_{12}= A_{21}= O(1- r^2) \quad \text{and} \quad A_{22}= e^a+ O(1- r^2),
\]
which shows that $A$ is coercive on $\Omega \cup \partial \D$.
\end{proof}

The quantity $\sqrt{g(a)} \big{(} H(a+ \xi)- H(a) \big{)}$, with $\xi \in \Cc^{2, \alpha}(\ov{\Omega})$, can be called a \emph{compactification} of the mean curvature of $X^a$ since it can be extended to the exterior boundary $\lac r= 1 \rac$ of $\Omega$. It is strongly linked with the compactification of the induced metric $g(a)$ by the following equality:
\[
A^{-1}= \matrix{2}{e^a & 0 \\ 0 & e^{-a}}+ O(1- r^2)= \frac{1}{\sqrt{|g(a)|}} g(a)+ O(1- r^2).
\]

\section[Moduli space of CMC-1/2 entire graphs]{Moduli space of CMC-$1/ 2$ entire graphs} \label{sec:graph}

In this section, we are interested in the subset $\Gg \subset \Ee$ of CMC-$1/ 2$ entire graphs contained in the half-space $\H^2 \times \R_+^*$. Since elements of $\Gg$ are simply connected, they can be \emph{globally} parametrized in graph coordinates at infinity over the whole disk $\D$ using~\eqref{eq:graphcoord}:
\[
X^{\eta}= \lp F, 2e^{\eta} \frac{1+ r^2}{1- r^2} \rp \quad \text{with} \quad F(r, \theta)= \frac{2r}{1+ r^2} e^{i\theta} \quad \text{and} \quad \eta \in \Cc^{2, \alpha}(\ov{\D}),
\]
and the geometrically defined function $\eta|_{\partial \D}: \S^1 \into \R$ is the \emph{value at infinity} of the surface.

\medskip

Consider a CMC-$1/ 2$ entire graph $S \in \Gg$, with graph coordinates at infinity $X^a$, where $a \in \Cc^{2, \alpha}(\ov{\D})$, and denote $\gamma^a= a|_{\partial \D}$ the value at infinity. A simple computation shows that the vertical component $\phi^a= \lan N^a, e_3 \ran$ of the upward pointing unit normal $N^a$ to $X^a$ can be expressed as:
\begin{equation} \label{eq:phia}
\phi^a= \frac{e^{-a}}{2c^a} \frac{1- r^2}{1+ r^2} \quad \text{with} \quad c^a= \frac{e^{-a}}{2} \frac{1- r^2}{1+ r^2} W(a),
\end{equation}
where $W(a)$ is given by~\eqref{eq:exprw} and $\phi^a= 1/ W(a)$. Note that $c^a$ is a positive function on $\ov{\D}$ such that $c^a|_{\partial \D}= 1/ 2$.

\medskip

In the sequel, we make the following abuse of notation denoting $H$ the operator:
\[
H: \eta \in \Cc^{2, \alpha}(\ov{\D}) \mapsto H(\eta) \in \Cc^{0, \alpha}(\ov{\D}),
\]
where $H(\eta)$ is the mean curvature of $X^{\eta}$, and calling it the \emph{mean curvature operator}.

\begin{lemma} \label{lem:diffh}

The differential of the operator $H$ at the point $a$ is:
\[
\a \eta \in \Cc^{2, \alpha}(\ov{\D}), \ DH(a) \cdot \eta= \frac{1}{2} L\lp \frac{\eta}{c^a} \rp,
\]
where $L$ is the Jacobi operator of $X^a$.

\end{lemma}

\begin{proof}
If $X^{\eta_t}$ is a differentiable family in the parameter $t$ such that $\eta_0= a$, it is a standard fact that:
\[
\lr \frac{d}{dt} \rb_{t= 0} H(\eta_t)= \frac{1}{2} L \bigg{\lan} \lr \frac{d}{dt} \rb_{t= 0} X^{\eta_t}, N^a \bigg{\ran}= \frac{1}{2} L \lp 2e^a \phi^a \frac{1+r^2}{1- r^2} \lr \frac{d\eta_t}{dt} \rb_{t= 0} \rp,
\]
and the expression~\eqref{eq:phia} of $\phi^a$ leads to the conclusion.
\end{proof}

Using Theorem~\ref{thm:compcourb}, we define the \emph{compactified mean curvature operator} to be:
\begin{equation} \label{eq:defcompcourb}
\ov{H}: \xi \in \Cc^{2, \alpha}(\ov{\D}) \mapsto \sqrt{|g(a)|} \lp H(a+ 2c^a \xi)- \frac{1}{2} \rp \in \Cc^{0, \alpha}(\ov{\D}).
\end{equation}
The \emph{compactified Jacobi operator} is $\ov{L}= D\ov{H}(0): \Cc^{2, \alpha}(\ov{\D}) \into \Cc^{0, \alpha}(\ov{\D})$ and using Lemma~\ref{lem:diffh} we know that:
\[
\ov{L}= \sqrt{|g(a)|} L.
\]

\begin{proposition}[Green identity] \label{prop:greeng}

For any $u, v \in \Cc^{2, \alpha}(\ov{\D})$, $\ov{L}$ satisfies the following identity:
\[
\int_{\ov{\D}} \lp u\ov{L} v- v\ov{L} u \rp d\ov{A}= \int_0^{2\pi} e^{-\gamma^a} \lr \lp u \der[v]{r}- v \der[u]{r} \rp \rb_{r= 1} d\theta,
\]
with $d\ov{A}$ the Lebesgue measure on $\ov{\D}$.

\end{proposition}

\begin{proof}
Let $u, v \in \Cc^{2, \alpha}(\ov{\D})$. For any $R \in (0, 1)$, $L$ satisfies a Green identity on $\lac r \leq R \rac$:
\[
\int_{\lac r \leq R \rac} (uLv- vLu) dA= \int_{\lac r= R \rac} \lp u\der[v]{\nu}- v\der[u]{\nu} \rp ds,
\]
where $dA$ and $ds$ are the measures corresponding to the metric induced by $X^a$ on $\lac r \leq R \rac$ and $\lac r= R \rac$ respectively, and where $\partial \cdot/ \partial \nu$ denotes the co-normal derivative. Notice that:
\begin{gather*}
dA= \sqrt{|g(a)|} \, d\ov{A}, \quad ds= \sqrt{g_{22}(a)} \, d\theta \\
\text{and} \quad \nu= \frac{1}{\sqrt{g_{22}(a) |g(a)|}} \big{(} g_{22}(a) X^a_1- g_{12}(a) X^a_2 \big{)},
\end{gather*}
with $d\ov{A}$ the Lebesgue measure on $\R^2$. Taking the limit when $R \into 1$, we obtain:
\[
\Lim{R}{1} \sqrt{g_{22}(a)} \der{\nu}= \Lim{R}{1} \lp \frac{g_{22}(a)}{\sqrt{|g(a)|}} \der{r}- \frac{g_{12}(a)}{\sqrt{|g(a)|}} \der{\theta} \rp= e^{-\gamma^a} \lr \der{r} \rb_{r= 1},
\]
and the identity follows.
\end{proof}

\begin{corollary} \label{cor:nosolg}

There is no solution $u \in \Cc^{2, \alpha}(\ov{\D})$ to the equation:
\[
\lac \begin{array}{ll}
\ov{L} u= 0         & \text{on } \ov{\D} \\
u|_{\partial \D}= 1
\end{array} \rr.
\]

\end{corollary}

\begin{proof}
By contradiction, suppose such a $u$ exist and apply Proposition~\ref{prop:greeng} to $\phi^a$ and $u$:
\begin{align*}
0 & = \int_{\ov{\D}} \lp \phi^a \ov{L} u- u\ov{L} \phi^a \rp d\ov{A}= \int_0^{2\pi} e^{-\gamma^a} \lr \lp \phi^a \der[u]{r}- u \der[\phi^a]{r} \rp \rb_{r= 1} d\theta \\
  & = \int_0^{2\pi} e^{-2\gamma^a} d\theta,
\end{align*}
since:
\[
\phi^a|_{r= 1}= 0 \quad \text{and} \quad \lr \der[\phi^a]{r} \rb_{r= 1}= \lr \lp -\frac{2re^{-a}}{1+ r^2}+ O(1- r^2) \rp \rb_{r= 1}= -e^{-\gamma^a}.
\]
This is impossible.
\end{proof}

Let $\ov{L}_0$ be the restriction of $\ov{L}$ to $\Cc^{2, \alpha}_0(\ov{\D})$ and $K= \ker \ov{L}_0$. Using the standard inclusions $\Cc^{2, \alpha}_0(\ov{\D}) \subset \Cc^{0, \alpha}(\ov{\D}) \subset L^2(\D)$, we denote by $K^{\bot}$ the orthogonal to $K$ in $\Cc^{0, \alpha}(\ov{\D})$ for the natural scalar product of $L^2(\D)$ and $K_0^{\bot}= K^{\bot} \cap \Cc^{2, \alpha}_0(\ov{\D})$.

It is a standard fact that the restriction $\ov{L}_0$ is a Fredholm operator with index zero (see for instance \cite{GiTr}). Namely $K= \R \phi^a$ and $\ov{L}_0 \big{(} \Cc^{2, \alpha}_0(\ov{\D}) \big{)}= K^{\bot}$.

\subsection{General deformations} \label{subsec:gendef}

Let $\mu_a: \Cc^{2, \alpha}(\S^1) \into \Cc^{2, \alpha}(\ov{\D})$ be the operator such that $\mu_a(\gamma)$ is the harmonic function on $\ov{\D}$ (for the flat laplacian) with value $\gamma- \gamma^a$ on the boundary $\partial \D$. In the sequel, we make constant use of the decomposition of $\Cc^{2, \alpha}(\ov{\D})$ in $\Cc^{2, \alpha}(\S^1) \times \R \times K_0^{\bot}$, which we call for short the \emph{decomposition induced} by $X^a$ or $a$, meaning that any $\eta \in \Cc^{2, \alpha}(\ov{\D})$ is characterized by a triple $(\gamma, \lambda, \sigma) \in \Cc^{2, \alpha}(\S^1) \times \R \times K_0^{\bot}$ such that:
\[
\eta= a+ 2c^a \big{(} \mu_a(\gamma)+ \lambda \phi^a+ \sigma \big{)}.
\]

Denote $\Pi_K$ and $\Pi_{K^{\bot}}$ be the orthogonal projections on $K$ and $K^{\bot}$ respectively. Following B. White \cite{Wh}, we show:

\begin{lemma} \label{lem:isomg}

Consider the map $\Phi: \Cc^{2, \alpha}(\S^1) \times \R \times K_0^{\bot} \into K^{\bot}$ defined by:
\[
\Phi(\gamma, \lambda, \sigma)= \Pi_{K^{\bot}} \circ \ov{H} \big{(} \mu_a(\gamma)+ \lambda \phi^a+ \sigma \big{)}.
\]
Then $D_3 \Phi(\gamma^a, 0, 0): K_0^{\bot} \into K^{\bot}$ is an isomorphism.
\end{lemma}

\begin{proof}
A direct computation gives $D_3 \Phi(\gamma^a, 0, 0)= \Pi_{K^{\bot}} \circ \ov{L}_0|_{K_0^{\bot}}$ and we know $K^{\bot}$ is the range of $\ov{L}_0$, which means $D_3 \Phi(\gamma^a, 0, 0): K_0^{\bot} \into K^{\bot}$ is an isomorphism.
\end{proof}

Therefore, we can apply the implicit function theorem to $\Phi$, which states that there exist an open neighborhood $U_a$ of $(\gamma^a, 0)$ in $\Cc^{2, \alpha}(\S^1) \times \R$ and a unique smooth map $\sigma: U_a \into K_0^{\bot}$ such that:
\[
\a (\gamma, \lambda) \in U_a, \ \Phi \big{(} \gamma, \lambda, \sigma(\gamma, \lambda) \big{)}= 0.
\]
Then we define the smooth maps $\xi_a: U_a \into \Cc^{2, \alpha}(\ov{\D})$, $\eta_a: U_a \into \Cc^{2, \alpha}(\ov{\D})$ and $\kappa_a: U_a \into K$ by:
\begin{gather*}
\xi_a(\gamma, \lambda)= \mu_a(\gamma)+ \lambda \phi^a+ \sigma(\gamma, \lambda), \quad \eta_a(\gamma, \lambda)= a+ 2c^a \xi_a(\gamma, \lambda) \\
\text{and} \quad \kappa_a(\gamma, \lambda)= \Pi_K \circ \ov{H} \big{(} \xi_a(\gamma, \lambda) \big{)}.
\end{gather*}
If a surface in $\Ee$, defined on $\D$, admits $X^{\eta_a(\gamma, \lambda)}$ as graph coordinates at infinity, we say that $\lac \gamma, \lambda \rac$ are the \emph{data} of the surface with respect to $S$ or to $a$.

\begin{lemma} \label{lem:etag}

The maps $\eta_a$ and $\xi_a$ have the following properties:

\begin{enumerate}

	\item $\xi_a(\gamma^a, 0)= 0$ and $\eta_a(\gamma^a, 0)= a$.

	\item $\a (\gamma, \lambda) \in U_a, \ \eta_a(\gamma, \lambda)|_{\partial \D}= \gamma$.

	\item $D_2  \xi_a(\gamma^a, 0): \lambda \in \R \mapsto \lambda \phi^a \in \Cc^{2, \alpha}(\ov{\D})$.

\end{enumerate}

\end{lemma}

\begin{proof}
Point~1 comes from the definition of $\mu_a$ and from the uniqueness in the implicit function theorem. Point~2 is a direct computation:
\begin{align*}
\eta_a(\gamma, \lambda)|_{\partial \D} & = a|_{\partial \D}+ 2c^a|_{\partial \D} \big{(} \mu_a(\gamma)|_{\partial \D}+ \lambda \phi^a|_{\partial \D}+ \sigma(\gamma, \lambda)|_{\partial \D} \big{)} \\
                                       & = \gamma^a+ 2\frac{1}{2} \big{(} (\gamma- \gamma^a) \big{)}= \gamma.
\end{align*}
For Point~3, it is sufficient to show $D_2 \sigma(\gamma^a, 0)= 0$. To do so we compute:
\begin{align*}
0 & = \lr \frac{d}{dt} \rb_{t= 0} \Phi \big{(} \gamma^a, t, \sigma(\gamma^a, t) \big{)}= \Pi_{K^{\bot}} \circ \ov{L} \big{(} \phi^a+ D_2 \sigma(\gamma^a, 0) \cdot 1 \big{)} \\
  & = \Pi_{K^{\bot}} \circ \ov{L}_0 \big{(} \phi^a+ D_2 \sigma(\gamma^a, 0) \cdot 1 \big{)}= \Pi_{K^{\bot}} \circ \ov{L}_0 \big{(} D_2 \sigma(\gamma^a, 0) \cdot 1 \big{)} \\
  & = \ov{L}_0 \big{(} D_2 \sigma(\gamma^a, 0) \cdot 1 \big{)}.
\end{align*}
Hence, $D_2 \sigma(\gamma^a, 0) \cdot 1 \in K \cap K_0^{\bot}= \lac 0 \rac$, which means $D_2 \sigma(\gamma^a, 0)= 0$.
\end{proof}

\begin{remark}

Consider $S, S' \in \Gg$ admitting respectively $X^a, X^{a'}$ as graph coordinates at infinity and suppose there exist a surface in $\Ee$ with data $\lac \gamma, \lambda \rac$ and $\lac \gamma', \lambda' \rac$ with respect to $S$ and $S'$ respectively. Therefore, this surface admits graph coordinates at infinity $X^{\eta_a(\gamma, \lambda)}$ and $X^{\eta_{a'}(\gamma', \lambda')}$ ---~i.e. $\eta_a(\gamma, \lambda)= \eta_{a'}(\gamma', \lambda')$~--- and we get:
\begin{equation} \label{eq:lienbasg}
\gamma'= \gamma \quad \text{and} \quad \lambda'= \frac{1}{|\phi^{a'}|_{L^2(\D)}^2} \bigg{\lan} \frac{\eta_a(\gamma, \lambda)- a'}{2c^{a'}}- \mu_{a'}(\gamma), \phi^{a'} \bigg{\ran}_{L^2(\D)}.
\end{equation}
The identity on values at infinity comes from Lemma~\ref{lem:etag} Point~2, and the expression of $\lambda'$ is just the projection along $\phi^{a'}$.

Note that a converse to this decomposition is the subject of Theorem~\ref{thm:rigidinf}, namely if $X^{\eta_a(\gamma, \lambda)}$ admits data with respect to $S'$, these data are $\lac \gamma', \lambda' \rac$ as defined in \eqref{eq:lienbasg}.

\end{remark}

Lemma~\ref{lem:etag} Point~2 also shows that the value at infinity of a surface $X^{\eta_a(\gamma, \lambda)}$ does not depend on $\lambda$, which means that given a value at infinity $\gamma$ there exists a $1$-parameter family of surfaces all with value at infinity equals to $\gamma$ as we show next.

\begin{proposition} \label{prop:lambdag}

Let $(\gamma, \lambda) \in U_a$. The surface $X^{\eta_a(\gamma, \lambda')}$ exists for any $\lambda' \in \R$ and coincides with $X^{\eta_a(\gamma, \lambda)}$ up to a vertical translation.

\end{proposition}

\begin{proof}
To ease the writing, denote $\wt{a}= \eta_a(\gamma, \lambda)$, $h(\wt{a})$ the height function of $X^{\wt{a}}$ i.e.:
\[
h(\wt{a})= 2e^{\wt{a}} \frac{1+ r^2}{1- r^2},
\]
and $m> 0$ the minimum of $h(\wt{a})$ on $\D$. We know from Proposition~\ref{prop:invisom} that we can parametrize a vertical translate of $X^{\wt{a}}$, by some $t \in \R$, by graph coordinates $X^{a'(t)}$ defined on $\D$ if and only if $t> -m$ and in that case:
\[
a'(t)= \wt{a}+ \log \lp 1+ t\frac{e^{-\wt{a}}}{2} \frac{1- r^2}{1+ r^2} \rp= \wt{a}+ \log \lp 1+ \frac{t}{h(\wt{a})} \rp.
\]
We also know that $a'(t)|_{\partial \D}= \wt{a}|_{\partial \D}$, which implies $\mu_a(\gamma^{a'(t)})= \mu_a(\gamma)$. Writing:
\[
a'(t)= a+ 2c^a \big{(} \mu_a(\gamma)+ \lambda'(t) \phi^a+ \sigma'(t) \big{)} \quad \text{with} \quad \lambda'(t) \in \R \quad \text{and} \quad \sigma'(t) \in K_0^{\bot},
\]
we only have to show that $\lambda'(t)$ is a bijection in the variable $t$ from the interval $(-m, +\ii)$ of possible translations onto $\R$. We have:
\[
\frac{a'(t)- a}{2c^a}= \frac{a'(t)- \wt{a}}{2c^a}+ \frac{\wt{a}- a}{2c^a}= \frac{1}{2c^a} \log \lp 1+ \frac{t}{h(\wt{a})} \rp+ \xi_a(\gamma, \lambda),
\]
and using \eqref{eq:lienbasg}, the expression of $\lambda'(t)$ is:
\begin{align*}
\lambda'(t) & = \lambda+ \frac{1}{2\pi |\phi^a|_{L^2(\D)}^2} \int_{\D} \frac{\phi^a}{c^a} \log \lp 1+ \frac{t}{h(\wt{a})} \rp \\
            & = \lambda+ \frac{1}{2\pi |\phi^a|_{L^2(\D)}^2} \int_{\D} (\phi^a)^2 h(a) \log \lp 1+ \frac{t}{h(\wt{a})} \rp \quad \text{since} \quad \frac{1}{c^a}= \phi^a h(a).
\end{align*}
Compute:
\[
\frac{d\lambda'(t)}{dt}= \frac{1}{2\pi |\phi^a|_{L^2(\D)}^2} \int_{\D} \frac{(\phi^a)^2 h(a)}{t+ h(\wt{a})}> 0
\]
i.e. $\lambda'(t)$ is a strictly increasing bijection from $(-m, +\ii)$ into $\R$. Also:
\[
\lambda'(t) \stackrel{(t \leq 0)}{\leq} \lambda+ \lc \frac{1}{2\pi |\phi^a|_{L^2(\D)}^2} \int_{\D} (\phi^a)^2 h(a) \rc \log \lp 1+ \frac{t}{m} \rp \tend{t}{-m} -\ii.
\]
If $M> 0$ is the maximum of $h(\wt{a})$ on the disk $\lac 0 \leq r \leq 1/ 2 \rac$, we get:
\begin{align*}
\lambda'(t) & \stackrel{(t \geq 0)}{\geq} \lambda+ \frac{1}{2\pi |\phi^a|_{L^2(\D)}^2} \int_{\lac 0 \leq r \leq 1/ 2 \rac} (\phi^a)^2 h(a) \log \lp 1+ \frac{t}{h(a)} \rp \\
            & \stackrel{\hphantom{(t \geq 0)}}{\geq} \lambda+ \lc \frac{1}{2\pi |\phi^a|_{L^2(\D)}^2} \int_{\lac 0 \leq r \leq 1/ 2 \rac} (\phi^a)^2 h(a) \rc \log \lp 1+ \frac{t}{M} \rp \tend{t}{+\ii} +\ii,
\end{align*}
which ensures that $\lambda'(t)$ is bijective from $(-m, +\ii)$ \emph{onto} $\R$.
\end{proof}

\subsection{CMC-1/2 deformations}

The values of the mean curvature of deformations $X^{\eta_a(\gamma, \lambda)}$ of $S$ are determined by $\kappa_a$. Indeed, for $(\gamma, \lambda) \in U_a$ we have $\Phi \big{(} \gamma, \lambda, \sigma(\gamma, \lambda) \big{)}= 0$ and:
\begin{equation} \label{eq:hkappa}
\ov{H} \big{(} \xi_a(\gamma, \lambda) \big{)}= \kappa_a(\gamma, \lambda)+ \Phi \big{(} \gamma, \lambda, \sigma(\gamma, \lambda) \big{)}= \kappa_a(\gamma, \lambda).
\end{equation}
In particular:
\[
\a (\gamma, \lambda) \in U_a, \ H\big{(} \eta_a(\gamma, \lambda) \big{)}= \frac{1}{2} \eq \kappa_a(\gamma, \lambda)= 0.
\]
Consider $\Uu_a= \kappa_a^{-1}(\lac 0 \rac) \cap U_a$. Using Proposition~\ref{prop:lambdag}, we can take $\Uu_a= \Gamma_a \times \R$ with $\Gamma_a$ a subset of $\Cc^{2, \alpha}(\S^1)$. Furthermore, since the construction is local, we can suppose $\Gamma_a$ connected.

\begin{proposition} \label{prop:localg}

$\Gamma_a$ is a codimension $1$ smooth submanifold of $\Cc^{2, \alpha}(\S^1)$. The tangent space to $\Gamma_a$ at $\gamma^a$ is the orthogonal space $\lan e^{-2\gamma^a} \ran^{\bot}$ to $e^{-2\gamma^a}$ in $\Cc^{2, \alpha}(\S^1)$ for the scalar product of $L^2(\S^1)$ and $\Gamma_a$ is a subset of:
\[
\lac \gamma \in \Cc^{2, \alpha}(\S^1) \lb |e^{-\gamma}|_{L^2(\S^1)}= 1 \rr \rac.
\]

\end{proposition}

\begin{proof}
We first show that $\kappa_a$ is a submersion at $(\gamma^a, 0)$. Using \eqref{eq:hkappa}, compute:
\begin{align*}
D_2 \kappa_a(\gamma^a, 0) \cdot 1 & = \lr \frac{d}{dt} \rb_{t= 0} \kappa_a(\gamma^a, t)= \lr \frac{d}{dt} \rb_{t= 0} \ov{H} \big{(} \xi_a(\gamma^a, t) \big{)} \\
                                  & = \ov{L} \big{(} D_2 \xi_a(\gamma^a, 0) \cdot 1 \big{)}= \ov{L}_0(\phi^a)= 0,
\end{align*}
since $\phi^a \in K$. Remains to find $\gamma \in \Cc^{2, \alpha}(\S^1)$ such that $D_1 \kappa_a(\gamma^a, 0) \cdot \gamma$ is not identically zero. We can take $\gamma= 1$. Indeed, using \eqref{eq:hkappa}:
\begin{align*}
D_1 \kappa_a(\gamma^a, 0) \cdot 1 & = \lr \frac{d}{dt} \rb_{t= 0} \kappa_a (\gamma^a+ t, 0)= \lr \frac{d}{dt} \rb_{t= 0} \ov{H} \big{(} \xi_a(\gamma^a+ t, 0) \big{)} \\
                                  & = \ov{L} \big{(} D_1 \xi_a(\gamma^a, 0) \cdot 1 \big{)} \neq 0,
\end{align*}
using Corollary~\ref{cor:nosolg} with $\big{(} D_1 \xi_a(\gamma^a, 0) \cdot 1 \big{)}|_{\partial \D}= 1$ deduced from Lemma~\ref{lem:etag} Point~2. Since $D\kappa_a$ is continuous and non zero at $(\gamma^a, 0)$, there exists an open neighborhood of $(\gamma^a, 0)$ in $\Cc^{2, \alpha}(\S^1) \times \R$ on which $\kappa_a$ is a submersion. Therefore, up to a restriction on $\Gamma_a$, we can suppose $\kappa_a$ is a submersion on $\Gamma_a \times \lac 0 \rac$, which implies $\Gamma_a$ is a submanifold of $\Cc^{2, \alpha}(\S^1)$ of codimension $1$.

Consider a smooth path $\gamma_t$ in $\Gamma_a$ with $\gamma_0= \gamma^a$ and tangent vectors $\dot{\gamma}_t$. Note that similarly:
\[
0= D\kappa_a(\gamma^a, 0) \cdot (\dot{\gamma}_0, 0)= \lr \frac{d}{dt} \rb_{t= 0} \kappa_a(\gamma_t, 0)= \ov{L} \big{(} D_1 \xi_a(\gamma^a, 0) \cdot \dot{\gamma}_0 \big{)}.
\]
Denote $v= D_1 \xi_a(\gamma^a, 0) \cdot \dot{\gamma}_0 \in \ker \ov{L}$. Knowing that:
\[
\phi^a|_{r= 1}= 0, \quad \lr \der[\phi^a]{r} \rb_{r= 1}= -e^{-\gamma^a} \quad \text{and} \quad v|_{r= 1}= \dot{\gamma}_0,
\]
apply Proposition~\ref{prop:greeng} to $\phi^a$ and $v$:
\begin{align} \label{eq:tangentg}
0 & = \int_{\ov{\D}} \lp \phi^a \ov{L} v- v\ov{L} \phi^a \rp d\ov{A}= \int_0^{2\pi} e^{-\gamma^a} \lr \lp \phi^a \der[v]{r}- v\der[\phi^a]{r} \rp \rb_{r= 1} d\theta \nonumber \\
  & = \int_0^{2\pi} \dot{\gamma}_0 e^{-2\gamma^a} d\theta= 2\pi \lan \dot{\gamma}_0, e^{-2\gamma^a} \ran_{L^2(\S^1)}.
\end{align}
Thus $\lan e^{-2\gamma^a} \ran^{\bot}$ is the tangent space to $\Gamma_a$ at $\gamma^a$, since it is of codimension $1$.

The stated inclusion for $\Gamma_a$ expresses the nullity of the vertical flux of an entire graph. If $ \gamma \in \Gamma_a$ is the value at infinity of a surface $S' \in \Gg$, consider the subset $V_R$, for $R \in (0, 1)$, of $\H^2 \times \R$ inside the vertical cylinder $C_R$ of (euclidean) radius $|F(R, \cdot)|= 2R/ (1+ R^2)$ in the Poincar{\'e} disk model~\eqref{eq:modpoinc}, delimited below by the slice $\H^2 \equiv \H^2 \times \lac 0 \rac$ and above by the surface $S'$. Since $e_3$ is a Killing vector field, using Stokes theorem we have:
\[
0= \int_{V_R} \Div e_3= \int_{S'_R} \lan N^{a'}, e_3 \ran+ \int_{D_R} \lan -e_3, e_3 \ran,
\]
where $S'_R$ is the part of $S'$ inside $C_R$ and $D_R$ the disk in $\H^2$ of (euclidean) radius $|F(R, \cdot)|= 2R/ (1+ R^2)$.

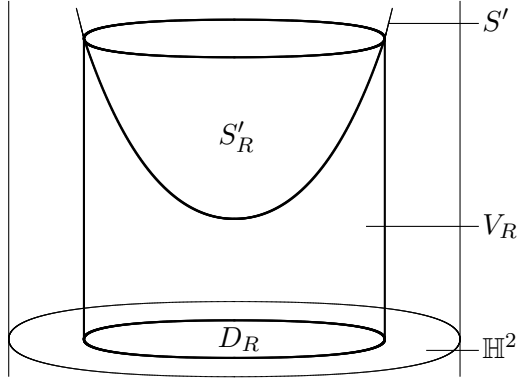
\begin{figure}[htbp]

\centering

\begin{picture}(6.8, 5)(-3, -0.5)

\put(-3, -0.5){\line(0, 1){5}} \put(3, -0.5){\line(0, 1){5}}
\qbezier(3, 0)(3, 0.5)(0, 0.5) \qbezier(0, 0.5)(-3, 0.5)(-3, 0) \qbezier(-3, 0)(-3, -0.5)(0, -0.5) \qbezier(0, -0.5)(3, -0.5)(3, 0)
\put(3.3, -0.25){$\H^2$} \put(3.25, -0.15){\line(-1, 0){0.7}} \put(3.3, 1.4){$V_R$} \put(3.25, 1.5){\line(-1, 0){1.5}} \put(3.3, 4.1){$S'$} \put(3.25, 4.2){\line(-1, 0){1.21}}
\ligne{2, 4}{2.1, 4.4} \ligne{-2, 4}{-2.1, 4.4}
\thicklines
\qbezier(0, 1.6)(1.2, 1.6)(2, 4) \qbezier(0, 1.6)(-1.2, 1.6)(-2, 4)
\qbezier(2, 0)(2, 0.25)(0, 0.25) \qbezier(0, 0.25)(-2, 0.25)(-2, 0) \qbezier(-2, 0)(-2, -0.25)(0, -0.25) \qbezier(0, -0.25)(2, -0.25)(2, 0)
\qbezier(2, 4)(2, 4.25)(0, 4.25) \qbezier(0, 4.25)(-2, 4.25)(-2, 4) \qbezier(-2, 4)(-2, 3.75)(0, 3.75) \qbezier(0, 3.75)(2, 3.75)(2, 4)
\put(-2, 0){\line(0, 1){4}} \put(2, 0){\line(0, 1){4}}
\put(-0.22, -0.1){$D_R$} \put(-0.21, 2.6){$S'_R$}

\end{picture}

\caption{Decomposition of the boundary of $V_R$}

\end{figure}

\nidt We use notations of Appendix~\ref{sec:comp}. If $X^{a'}$ are the graph coordinates at infinity of $S'$, we have $\Delta_{g(a')} X^{a'}= 2H(a') N^{a'}= N^{a'}$, since $H(a')= 1/ 2$, and the first integral writes:
\[
\int_{S'_R} \lan N^{a'}, e_3 \ran= \int_{\lac r \leq R \rac} \Delta_{g(a')} h(a')~dA= \int_{\lac r= R \rac} \der[h(a')]{\nu} ds,
\]
and we know that:
\[
\der{\nu}= \frac{1}{\sqrt{g_{22}(a')}} \lp \frac{g_{22}(a')}{\sqrt{|g(a')|}} \der{r}- \frac{g_{12}(a')}{\sqrt{|g(a')|}} \der{\theta} \rp \quad \text{and} \quad h(a')= 2e^{a'} \frac{1+ r^2}{1- r^2}.
\]
Using the expressions of $g_{12}(a')$, $g_{22}(a')$ and $|g(a')|$ computed in Appendix~\ref{sec:comp}, we get:
\[
\der[h(a')]{\nu}= \frac{1}{\sqrt{g_{22}(a')}} \frac{8r^2}{w(a')(1- r^2)^2} \lc 1+ \frac{(1+ r^2) a'_r}{4r} (1- r^2) \rc,
\]
and since $ds= \sqrt{g_{22}(a')} \, d\theta$, we obtain:
\begin{align*}
\int_{S'_R} \lan N^{a'}, e_3 \ran & = \frac{8R^2}{(1- R^2)^2} \int_0^{2\pi} \lr \frac{1}{w(a')} \lc 1+ \frac{(1+ r^2) a'_r}{4r} (1- r^2) \rc \rb_{r= R} d\theta \\
                                  & = \frac{16\pi R^2}{(1- R^2)^2}+ 2\pi \lp 1- |e^{-a'(R, \cdot)}|_{L^2(\S^1)} \rp+ O(1- R^2).
\end{align*}
The second is the area of $D_R$:
\[
\int_{D_R} 1= 2\pi \int_0^R \sqrt{|\sigma|} dr= 16\pi \int_0^{R^2} \frac{1+ r}{(1- r)^3} dr= \frac{16\pi R^2}{(1- R^2)^2}.
\]
Making $R \into 1$, we get $|e^{-\gamma}|_{L^2(\S^1)}= 1$ and the inclusion for $\Gamma_a$.
\end{proof}

A.~E.~Treibergs showed (see \cite{Tr}) that given a $\Cc^2$ curve ---~generalized to continuous curves by H.~I.~Choi and A.~E.~Treibergs in \cite{ChTr}~--- $\gamma: \S^1 \into \R$, there exists a CMC-$1/ 2$ complete entire vertical graph in the $3$-dimensional Minkowski space which is asymptotically at signed distance $\gamma$ from the light cone. Namely, it is the graph of a smooth function $f: \R^2 \into \R$ such that:
\[
f(x)= |x|+ \gamma \lp \frac{x}{|x|} \rp + \epsilon(x) \quad \text{with} \quad \Lim{|x|}{+\ii} \epsilon(x)= 0.
\]
Proposition~\ref{prop:localg} is indeed a $\Cc^{2, \alpha}$ local version of this result in $\H^2 \times \R$:

\begin{theorem} \label{thm:treibergsloc}

Let $\gamma \in \Cc^{2, \alpha}(\S^1)$, small in the $\Cc^{2, \alpha}$-norm, be such that $|e^{-\gamma}|_{L^2(\S^1)}= 1$. Then there exists a surface in $\Gg$ with $\gamma$ as value at infinity. In other words, there exists a CMC-$1/ 2$ complete entire vertical graph at asymptotic horizontal signed distance $\gamma$ from the hyperboloid $S_0$.

\end{theorem}

\begin{proof}
If $\gamma$ is sufficiently small in the $\Cc^{2, \alpha}$ norm, then $\gamma \in \Gamma_0$ and $X^{\eta_0(\gamma, 0)}$ is a CMC-$1/ 2$ entire graph admitting $\gamma$ as value at infinity.
\end{proof}

Another consequence of Proposition~\ref{prop:localg} is the global structure of $\Gg$:

\begin{theorem} \label{thm:structg}

The family $\Gg$ can be endowed with a structure of infinite dimensional smooth manifold.

\end{theorem}

\begin{proof}
First, the family $\Gg$ is non-empty since $S_0 \in \Gg$. Consider a surface $S \in \Gg$ with graph coordinates at infinity $X^a$ and $\Vv_a \subset \Gg$ the set of surfaces admitting data in $\Uu_a$. From uniqueness in the implicit function theorem we know that the map:
\[
\tau_a: S' \in \Vv_a \mapsto (\gamma, \lambda) \in \Uu_a,
\]
where $\lac \gamma, \lambda \rac$ are the data of $S'$ with respect to $a$, is a bijection. To prove that the couple $(\Vv_a, \tau_a)$ form a smooth atlas, it only remains to show that the transition maps are smooth. But identities~\eqref{eq:lienbasg} are precisely the transition map from $(\Vv_a, \tau_a)$ to $(\Vv_{a'}, \tau_{a'})$, which concludes the proof.
\end{proof}

\section{A half-space theorem}

In \cite{NeSE}, B.~Nelli and R.~S{\'a}~Earp show a half-space theorem for the hyperboloid $S_0$. We extend this result to the family $\Gg$ of CMC-$1/ 2$ entire graphs with appropriate graph coordinates at infinity. The proof is based on the idea of B.~Daniel, W.~H.~Meeks and H.~Rosenberg \cite{DaMeRo} in Heisenberg space. A key-ingredient is to construct a family of surfaces with boundary; our tool to do this is the following:

\begin{lemma} \label{lem:defstable}

Let $E$ be a CMC-$1/ 2$ surface with boundary admitting graph coordinates at infinity $X^a$ defined on an admissible domain $\Omega_R \in \Dd$, with $R \in (0, 1)$ and $a \in \Cc^{2, \alpha}(\ov{\Omega_R})$. Denote $\gamma^a_{int}= a|_{\lac r= R \rac}$ and $\gamma^a_{ext}= a|_{\partial \D}$. Then for any $(\gamma_{int}, \gamma_{ext})$ in a neighborhood of $(\gamma^a_{int}, \gamma^a_{ext})$ in $(\Cc^{2, \alpha}(\S^1))^2$, there exists a CMC-$1/ 2$ surface admitting graph coordinates at infinity $X^{a'}$ defined on $\Omega_R$ such that $a'|_{\lac r= R \rac}= \gamma_{int}$ and $a'|_{\partial \D}= \gamma_{ext}$.

\end{lemma}

\begin{proof}
Remark first that $E$ is strictly stable since the third coordinate $\phi^a$ of the upward pointing normal is a positive Jacobi function (see \cite{FCSc}). In particular, it means that the compactified Jacobi operator $\ov{L}_0$ restricted to $\Cc^{2, \alpha}_0(\ov{\Omega_R})$ is injective and, since it is a Fredholm operator of index zero, $\ov{L}_0$ is also surjective on $\Cc^{0, \alpha}(\ov{\Omega_R})$. In other words, using the same notations as of Section~\ref{sec:graph}, we get $K= \lac 0 \rac$, $K^{\bot}= \Cc^{0, \alpha}(\ov{\Omega_R})$, $K_0^{\bot}= \Cc^{2, \alpha}_0(\ov{\Omega_R})$ and $\ov{L}_0: \Cc^{2, \alpha}_0(\ov{\Omega_R}) \into \Cc^{0, \alpha}(\ov{\Omega_R})$ is an isomorphism.

Consider the map $\Phi: (\Cc^{2, \alpha}(\S^1))^2 \times \Cc^{2, \alpha}_0(\ov{\Omega_R}) \into \Cc^{0, \alpha}(\ov{\Omega_R})$ defined by:
\[
\Phi(\gamma_{int}, \gamma_{ext}, \sigma)= \ov{H} \big{(} \mu_a(\gamma_{int}, \gamma_{ext})+ \sigma \big{)},
\]
where $\ov{H}$ is the compactified mean curvature operator as defined in \eqref{eq:defcompcourb} and $\mu_a: (\Cc^{2, \alpha}(\S^1))^2 \into \Cc^{2, \alpha}(\ov{\Omega_R})$ is the operator such that $\mu_a(\gamma_{int}, \gamma_{ext})$ is the harmonic function on $\Omega_R$ with value
\[
\lp \frac{\gamma_{int}- \gamma^a_{int}}{2c^a|_{\lac r= R \rac}}, \gamma_{ext}- \gamma^a_{ext} \rp
\]
on the boundary of $\Omega_R$. The particular value at the interior boundary $\lac r= R \rac$ is made so that the boundary values of $2c^a \mu_a(\gamma_{int}, \gamma_{ext})$ are precisely $(\gamma_{int}- \gamma^a_{int}, \gamma_{ext}- \gamma^a_{ext})$. Hence, $D_3 \Phi(\gamma^a_{int}, \gamma^a_{ext}, 0): \Cc^{2, \alpha}_0(\ov{\Omega_R}) \into \Cc^{0, \alpha}(\ov{\Omega_R})$ is an isomorphism since $D_3 \Phi(\gamma^a_{int}, \gamma^a_{ext}, 0)= \ov{L}_0$ ---~for the same reason as in Lemma~\ref{lem:isomg}~--- and we can apply the implicit function theorem as in Section~\ref{subsec:gendef}. There exist a neighborhood $U$ of $(\gamma^a_{int}, \gamma^a_{ext})$ in $(\Cc^{2, \alpha}(\S^1))^2$ and a smooth map $\sigma: U \into \Cc^{0, \alpha}(\ov{\Omega_R})$ such that:
\[
\a (\gamma_{int}, \gamma_{ext}) \in U, \ \Phi \big{(} \gamma_{int}, \gamma_{ext}, \sigma(\gamma_{int}, \gamma_{ext}) \big{)}= 0.
\]
We can take $a'= a+ 2c^a \big{(} \mu_a(\gamma_{int}, \gamma_{ext})+ \sigma(\gamma_{int}, \gamma_{ext}) \big{)}$.
\end{proof}

We now can show the following half-space result:

\begin{theorem} \label{thm:halfspace}

Let $\Sigma$ be a CMC-$1/ 2$ surface which is properly immersed in $\H^2 \times \R$ and lies on one side of a CMC-$1/ 2$ entire graph $S \in \Gg$ admitting graph coordinates at infinity $X^a$ with $a \in \Cc^{2, \alpha}(\ov{\D})$. Suppose also that $\Sigma$ is well oriented with respect to $S$. Then $\Sigma$ coincides with $S$ up to a vertical translation.

\end{theorem}

By ``well oriented'', we mean that the mean curvature vector of $\Sigma$ points in the connected component of $\H^2 \times \R$ bounded by $S$ and $\Sigma$, so that classical maximum principle can apply. But since $\Sigma$ is not necessarily embedded, this condition has a meaning only for points of $\Sigma$ lying on the boundary of the connected component (see \cite[Section~4]{Ma}). For sake of simplicity, we suppose in the proof that $\Sigma$ is above $S$ ; the remaining case can be treated exactly the same way because of the orientation condition.

\medskip

\begin{proof}
Denote $T^c: \H^2 \times \R \into \H^2 \times \R$ the vertical translation by $c \in \R$ and:
\[
c_0= \inf \lac c \in \R \st[big] \Sigma \cap T^c(S) \neq \emptyset \rac.
\]
If $\Sigma \cap T^{c_0}(S) \neq \emptyset$ then by maximum principle, $\Sigma$ coincides with $T^{c_0}(S)$.

From now on, we suppose $\Sigma \cap T^{c_0}(S)= \emptyset$ and ---~up to a vertical translation~--- $c_0= 0$. In other words:
\[
\Sigma \cap S= \emptyset \quad \text{and} \quad \a c> 0, \ \Sigma \cap T^c(S) \neq \emptyset.
\]
We want to construct a CMC-$1/ 2$ surface with boundary intersecting $\Sigma$ in an interior point. To do so, consider $R \in (0, 1/2)$ and admissible domains $\Omega_R, \Omega_{2R} \in \Dd$. There exists $\delta> 0$ such that $\Sigma$ intersects $T^c(S)$ only inside the exterior domain $\Omega_{2R} \times \R$ for any $0< c < 2\delta$:
\[
\a c< 2\delta, \ \big{(} T^c(S) \cap \Sigma \big{)} \subset \Omega_{2R} \times \R.
\]
Denote $E= T^{\delta}(S) \cap (\Omega_R \times \R)$. $E$ is a CMC-$1/ 2$ surface with boundary. We can apply Lemma~\ref{lem:defstable} to deform $E$ and construct a family $\big{(} E(\epsilon) \big{)}_{\epsilon \geq 0}$ of CMC-$1/ 2$ surfaces with the same boundary and at prescribed distance at infinity from $E$. Namely:

\begin{itemize}

	\item $E(\epsilon)$ is at constant asymptotic horizontal signed distance $-\epsilon$ from $E$;

	\item $E(\epsilon)$ coincides with $E$ on the interior boundary $\lac |w|= R \rac \times \R$ of $\Omega_R \times \R$;

	\item $E(0)= E$.

\end{itemize}

\nidt Moreover, by the implicit function theorem applied in Lemma~\ref{lem:defstable},  we know that the family $E(\epsilon)$ is uniformly smooth. Hence, $E(\epsilon)$ converges to $E$ when $\epsilon \into 0$ and since $E \cap \Sigma \neq \emptyset$, there exists $\epsilon_0> 0$ such that $E(\epsilon_0) \cap \Sigma \neq \emptyset$.

At infinity $E(\epsilon_0)$ is outside $S$, thus $T^c \big{(} E(\epsilon_0) \big{)} \cap \Sigma= \emptyset$ for large $c< 0$. Consider:
\[
c_1= \sup \lac c< 0 \st[big] T^c \big{(} E(\epsilon_0) \big{)} \cap \Sigma= \emptyset \rac \leq 0.
\]
We know that $T^{c_1} \big{(} E(\epsilon_0) \big{)} \cap \Sigma \neq \emptyset$ since the first intersection point cannot be at infinity. And this intersection does not occur on the boundary of $T^{c_1} \big{(} E(\epsilon_0) \big{)}$, since the boundary lies outside $\Omega_{2R} \times \R$. Therefore, the first intersection point is interior to $T^{c_1} \big{(} E(\epsilon_0) \big{)}$ and by maximum principle, $\Sigma$ coincides with $T^{c_1} \big{(} E(\epsilon_0) \big{)}$ over $\Omega_R$, which is impossible.
\end{proof}

We can deduce from Theorem~\ref{thm:halfspace} a uniqueness result at infinity for the family $\Gg$:

\begin{theorem} \label{thm:rigidinf}

Let $S, S'$ be CMC-$1/ 2$ entire graphs in $\Gg$ admitting graph coordinates at infinity $X^a, X^{a'}$ respectively, with $a, a' \in \Cc^{2, \alpha}(\ov{\D})$. Suppose there exist a surface $\Sigma$ admitting data $(\gamma, \lambda) \in \Gamma_a \times \R$ with respect to $S$ and, as in \eqref{eq:lienbasg}, denote:
\[
\lambda'= \frac{1}{|\phi^{a'}|_{L^2(\D)}^2} \bigg{\lan} \frac{\eta_a(\gamma, \lambda)- a'}{2c^{a'}}- \mu_{a'}(\gamma), \phi^{a'} \bigg{\ran}_{L^2(\D)},
\]
with $X^{\eta_a(\gamma, \lambda)}$ the graph coordinates at infinity of $\Sigma$. Suppose $\gamma \in \Gamma_{a'}$, then $\Sigma$ admits data $\lac \gamma, \lambda' \rac$ with respect to $S'$; in other words, $\eta_a(\gamma, \lambda)= \eta_{a'}(\gamma, \lambda')$. 
\end{theorem}

\begin{proof}
We first make three remarks:

\begin{itemize}

	\item Since $\gamma \in \Gamma_{a'}$, the function $\eta_{a'}(\gamma, \lambda')$ exists.

	\item If $\Sigma$ admits data with respect to $S'$, then from \eqref{eq:lienbasg} and the definition of $\lambda'$ above, we know that the data are precisely $\lac \gamma, \lambda' \rac$. 

	\item To show that $\Sigma$ admits data with respect to $S'$, we only have to show that a vertical translate $T^{c_0}(\Sigma)$ of $\Sigma$, with $c_0 \geq 0$, admits data with respect to $S'$.

\end{itemize}

\nidt Consider graph coordinates at infinity $X^{\eta}$ for $T^c(\Sigma)$ with $c \geq 0$. Suppose there exist $r_0 \in (0, 1)$ such that the height functions of $T^c(\Sigma)$ and $X^{\eta_{a'}(\gamma, \lambda')}$ verify $h(\eta)> h \big{(} \eta_{a'}(\gamma, \lambda') \big{)}$ for any $(r, \theta) \in [r_0, 1) \times \S^1$. We take:
\[
c_0= c+ \max_{[0, r_0] \times \S^1} \lb h \big{(} \eta_{a'}(\gamma, \lambda') \big{)}- h(\eta) \rb,
\]
so that $T^{c_0}(\Sigma)$ is above $X^{\eta_{a'}(\gamma, \lambda')}$. Applying Theorem~\ref{thm:halfspace}, we deduce that $T^{c_0}(\Sigma)$ is a vertical translate of $X^{\eta_{a'}(\gamma, \lambda')}$, and hence admits data with respect to $S'$.

Remains to show the existence of $r_0$. Note that $h(\eta)> h \big{(} \eta_{a'}(\gamma, \lambda') \big{)}$ if and only if $\eta> \eta_{a'}(\gamma, \lambda')$, so that we only have to work with the ``eta's''. Since $X^{\eta}$ are the graph coordinates at infinity of $T^c(\sigma)$, we have:
\[
h(\eta)= h\big{(} \eta_a(\gamma, \lambda) \big{)}+ c \quad \text{i.e.} \quad \eta= \eta_a(\gamma, \lambda)+ \log \lp 1+ c\frac{e^{-\eta_a(\gamma, \lambda)}}{2} \frac{1- r^2}{1+ r^2} \rp.
\]
Moreover, by definition, $\eta_{a'}(\gamma, \lambda')$ writes:
\[
\eta_{a'}(\gamma, \lambda')= a'+ 2c^{a'} \big{[} \mu_{a'}(\gamma)+ \lambda' \phi^{a'}+ \sigma(\gamma, \lambda') \big{]},
\]
and we use the decomposition of $\Cc^{2, \alpha}(\ov{\D})$ induced by $X^{a'}$ to express $\eta_a(\gamma, \lambda)$:
\[
\eta_a(\gamma, \lambda)= a'+ 2c^{a'} \big{[} \mu_{a'}(\gamma)+ \lambda' \phi^{a'}+ \sigma \big{]}.
\]
In this expression, the argument of $\mu_{a'}$ is indeed $\eta_a(\gamma, \lambda)|_{\partial \D}= \gamma$, the factor of $\phi^{a'}$ is exactly $\lambda'$ by definition of $\lambda'$ and $\sigma \in \Cc^{2, \alpha}_0(\ov{\D})$ is orthogonal to $\phi^{a'}$. We compute:
\begin{align*}
\eta- \eta_{a'}(\gamma, \lambda') & = \lc \eta_a(\gamma, \lambda)+ \log \lp 1+ c\frac{e^{-\eta_a(\gamma, \lambda)}}{2} \frac{1- r^2}{1+ r^2} \rp \rc- \eta_{a'}(\gamma, \lambda') \\
                                  & = \underbrace{2c^{a'} \big{[} \sigma- \sigma(\gamma, \lambda') \big{]}}_{\in \Cc^{2, \alpha}_0(\ov{\D})}+ \log \lp 1+ c\frac{e^{-\eta_a(\gamma, \lambda)}}{2} \frac{1- r^2}{1+ r^2} \rp,
\end{align*}
which gives:
\[
\lr \der{r} \big{(} \eta- \eta_{a'}(\gamma, \lambda') \big{)} \rb_{\partial \D}= 2\lr \der{r} \Big{(} c^{a'} \big{[} \sigma- \sigma(\gamma, \lambda') \big{]} \Big{)} \rb_{\partial \D}- c\frac{e^{-\gamma}}{2}.
\]
Taking:
\[
c= 2\max_{\partial \D} e^{\gamma} \lp 1+ 2\der{r} \Big{(} c^{a'} \big{[} \sigma- \sigma(\gamma, \lambda') \big{]} \Big{)} \rp,
\]
we get: 
\[
\lr \der{r} \big{(} \eta- \eta_{a'}(\gamma, \lambda') \big{)} \rb_{\partial \D} \leq -1.
\]
Thus, there exists $r_0 \in (0, 1)$ such that $\eta- \eta_{a'}(\gamma, \lambda')$ is strictly decreasing (in $r$) on $[r_0, 1) \times \S^1$. It means that $\eta- \eta_{a'}(\gamma, \lambda')> 0$ on $[r_0, 1) \times \S^1$ since $\eta- \eta_{a'}(\gamma, \lambda')$ is identically zero on $\partial \D$, which concludes the proof.
\end{proof}

\section[Deformations of CMC-1/2 annuli]{Deformations of CMC-$1/ 2$ annuli} \label{sec:ann}

R.~S{\'a}~Earp and E.~Toubiana showed in \cite{SETo} that ---~up to a not necessarily orientation preserving isometry of $\H^2 \times \R$~--- a rotational CMC-$1/ 2$ vertical annulus is a bigraph, symmetrical with respect to the slice $\H^2 \times \lac 0 \rac$. The upper graph part of such an annulus admits graph coordinates at infinity $(F, \bar{h}_{\beta})$, with $\beta$ a positive real number, $\beta \neq 1$ and $\bar{h}_{\beta}$ defined by:
\[
\bar{h}_{\beta}(r)= \int_{|\log \beta|}^{2\log \lp \frac{1+ r}{1- r} \rp} \frac{\cosh t- \beta}{\sqrt{2\beta \cosh t- 1- \beta^2}} dt \quad \text{where} \quad r \geq \lb \frac{1- \sqrt{\beta}}{1+ \sqrt{\beta}} \rb= R_{\beta}.
\]
We denote by $A_{\beta}$ this annulus, which is embedded if $0< \beta< 1$ and only immersed when $\beta> 1$.

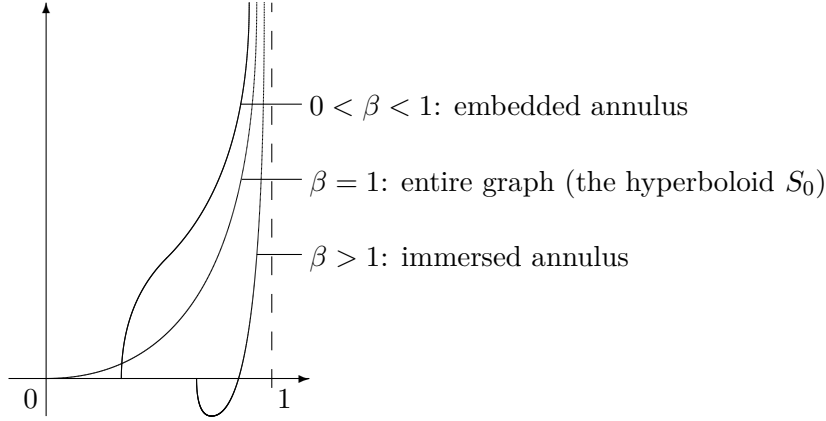
\begin{figure}[htbp]

\centering

\begin{picture}(10.9, 6)(-0.5, -0.5)

\put(-0.5, 0){\vector(1, 0){4}} \put(0, -0.5){\vector(0, 1){5.5}}
\multiput(3, -0.125)(0, 0.5){10}{\line(0, 1){0.25}} \put(3, 4.875){\line(0, 1){0.125}}
\put(-0.3, -0.39){$0$} \put(3.07, -0.39){$1$}
\qbezier(1, 0)(1, 1)(1.6, 1.6) \qbezier(1.6, 1.6)(2.7, 2.7)(2.7, 5) 
\qbezier(0, 0)(2.8, 0)(2.8, 5)
\qbezier(2, 0)(2, -0.5)(2.2, -0.5) \qbezier(2.2, -0.5)(2.9, -0.5)(2.9, 5)
\put(3.5, 3.5){$0< \beta< 1$: embedded annulus} \put(3.4, 3.65){\line(-1, 0){0.82}}
\put(3.5, 2.5){$\beta= 1$: entire graph (the hyperboloid $S_0$)} \put(3.4, 2.65){\line(-1, 0){0.8}}
\put(3.5, 1.5){$\beta> 1$: immersed annulus} \put(3.4, 1.65){\line(-1, 0){0.6}}

\end{picture}

\caption{Profile curves of rotational CMC-$1/ 2$ examples in the Poincar{\'e} disk model~\eqref{eq:modpoinc}}

\end{figure}

\nidt We have the following asymptotic development as $r \into 1$:
\begin{equation} \label{eq:asdev}
\bar{h}_{\beta}(r)= \frac{1}{\sqrt{\beta}} \frac{1+ r}{1- r}+ O(1),
\end{equation}
which means that the restriction of $(F, \bar{h}_{\beta})$ to the exterior domain $\Omega_{R_{\beta}}$ is in $\Ee$ with constant value $-\log \beta$ at infinity. Therefore, the method developed in Section~\ref{sec:graph} should adapt to the study of deformations of these annuli.

\medskip

For our purpose, we slightly change the notations. Fix $\beta> 0$ with $\beta \neq 1$; the annulus $A_{\beta}$ is now the model surface. To deform rotational annuli, we need conformal coordinates to provide a compactification of the mean curvature. The function $\psi: [R_{\beta}, 1] \into \R$ defined by:
\[
\a x \in [R_{\beta}, 1], \ \psi(x)= \frac{4}{|1- \beta|} \int_{R_{\beta}}^x \frac{dt}{\sqrt{(t^2- R_{\beta}^2)(R_{\beta}^{-2}- t^2)}}
\]
is an increasing bijection onto $[0, T]$ where:
\[
T= \frac{4}{|1- \beta|} \int_{R_{\beta}}^1 \frac{dt}{\sqrt{(t^2- R_{\beta}^2)(R_{\beta}^{-2}- t^2)}}.
\]
If $r: [-T, T] \into [R_{\beta}, 1]$ denotes the even function such that $r|_{[0, T]}= \psi^{-1}$, then a conformal parametrization of the annulus $A_{\beta}$, written in cylindrical coordinates, is the following:
\begin{equation} \label{eq:paramconfa}
X^0: (s, \theta) \in \Omega^{\beta} \mapsto \Big{(} F \big{(} r(s), \theta \big{)}, h_{\beta}(s) \Big{)} \in \H^2 \times \R,
\end{equation}
where $\Omega^{\beta}= (-T, T) \times \S^1$ and $h_{\beta}: (-T, T) \into \R$ denotes the odd extension of $\bar{h}_{\beta} \circ r$ i.e.:
\[
h_{\beta}(s)= \lac \begin{array}{ll}
\bar{h}_{\beta} \big{(} r(s) \big{)}  & \text{if } s \geq 0 \\
-\bar{h}_{\beta} \big{(} r(s) \big{)} & \text{if } s \leq 0
\end{array} \rr.
\]
We also identify functions over $A_{\beta}$ with functions over $\Omega^{\beta}$. The \emph{cylindrical parametrization} of a deformed annulus is the following immersion:
\begin{equation} \label{eq:cylparam}
X^{\eta}: (s, \theta) \in \Omega^{\beta} \mapsto \Big{(} e^{-\chi(s) \eta(s, \theta)} F \big{(} r(s), \theta \big{)}, e^{\big{(} 1 -\chi(s) \big{)} \eta(s, \theta)} h_{\beta}(s) \Big{)},
\end{equation}
where $\eta \in \Cc^{2, \alpha}(\ov{\Omega^{\beta}})$ and $\chi: [-T, T] \into [0, 1]$ is a smooth increasing even function such that $\chi|_{[0, T/ 3]} \equiv 1$ and $\chi|_{[2T/ 3, T]} \equiv 0$. The determinant of the first fundamental form is $|g(\eta)|$, the mean curvature $H(\eta)$ and the values at infinity are the couple $\big{(} \eta(T, \cdot), \eta(-T, \cdot) \big{)} \in (\Cc^{2, \alpha}(\S^1))^2$.

\begin{remark}

Contrary to the case of entire graphs, the deformation here is not only vertical (on the third coordinate). It is due to the fact that a vertical deformation would be tangent to the rotational annulus on the intersection of the annulus with its horizontal symmetry plane ---~corresponding to $s= 0$ in the cylindrical coordinates.

\end{remark}

\subsection{Non degeneracy of rotational annuli}

As in Section~\ref{sec:graph}, we need to understand the Jacobi functions in order to control the deformations. Thus, we focus the study on annuli in $\Ee$ that are non degenerate in the following sense:

\begin{definition}

A surface in $\Ee$ is said to be \emph{non degenerate} if the only Jacobi functions that are zero at infinity on each end of the surface (i.e. when $r= 1$ in the graph coordinates at infinity of the ends) come from isometries of $\H^2 \times \R$.

\end{definition}

\begin{remark} \label{rk:nondegen}

A direct consequence of the proof of Proposition~\ref{prop:invisom} and the shape of the ends is that if an annulus in $\Ee$ is non degenerate, then the space of Jacobi functions which are zero on the boundary is $1$-dimensional, generated by the vertical component of the unit normal. Another fact is that, since the rank of the Jacobi operator is locally constant, small deformations of a non degenerate annulus are still non degenerate. Therefore, the method used in Section~\ref{sec:graph} can be strictly transposed to the study of deformations in a small neighborhood of a non degenerate example.

\end{remark}

\begin{proposition} \label{prop:nondegen}

The annulus $A_{\beta}$ is non degenerate for any value of $\beta \ (\neq 1)$.

\end{proposition}

\begin{proof}
If $L$ denotes the Jacobi operator of $A_{\beta}$, the compactified Jacobi operator $\ov{L}= \sqrt{|g(0)|} L$ of $A_{\beta}$ can be written $\Delta+ q(s)$ in the conformal parametrization~\eqref{eq:paramconfa}, with $\Delta$ the flat laplacian and $q \in \Cc^0([-T, T])$. Moreover, $A_{\beta}$ being symmetric with respect to $\H^2 \times \lac 0 \rac$, the function $q$ is even.

Since a Jacobi function is $2\pi$-periodic in $\theta$, using the Fourier decomposition, we reduce the problem to solving a family~\eqref{eq:dn} of Dirichlet problems on $\Cc^2([-T, T])$ for $n \in \N$:
\begin{equation} \label{eq:dn}
\lac \begin{array}{l}
u''+ \big{(} q(s)- n^2 \big{)} u= 0 \\
u(-T)= u(T)= 0
\end{array} \rr. \tag{D$_n$}
\end{equation}
We make two immediate observations:

\begin{itemize}

	\item Considering a solution of~\eqref{eq:dn} for any $n \in \N$, its odd and even parts are also solutions of~\eqref{eq:dn}. Hence, we only have to consider odd and even solutions.

	\item The vertical component $\phi$ of the unit normal to $A_{\beta}$ is an odd solution of (D$_0$) which does not vanish on $(0, T)$.

\end{itemize}

Let $n \in \N$. \emph{An odd solution of~\eqref{eq:dn} is proportional to $\phi$}. Otherwise, using Sturm comparison theorem with $q- n^2 \leq q$, $\phi$ should vanish once in $(0, T)$. \emph{There is no even solution to~\eqref{eq:dn}}. Suppose such a function $u$ exist. Using Sturm comparison theorem, $u$ vanishes nowhere in $(-T, T)$, which means $n^2$ is the first eigenvalue of the elliptic operator:
\[
\frac{d^2}{ds^2}+ q(s),
\]
and that the corresponding eigenspace is one dimensional equal to $\R u$. Moreover $\phi$ is an eigenfunction of this operator associated to the eigenvalue $0$, which implies $n= 0$ and $\phi= \lambda u$ for some $\lambda \in \R$. But $\phi$ is odd, which is a contradiction.
\end{proof}

\subsection{Deformations of rotational annuli}

The method of Section~\ref{sec:graph} can be strictly transposed to deform annuli ---~not necessarily rotational~--- verifying some technical conditions. We restrict ourselves to the construction of deformations of the rotational example $A_{\beta}$ for sake of clarity and since we need only these deformations to prove Theorem~\ref{thm:nonalign}; we refer to Remark~\ref{rk:deformnonrot} for the key conditions of the generalization.

\begin{lemma}

The differential of th operator $H$ at $0$ writes:
\[
\a \eta \in \Cc^{2, \alpha}(\ov{\Omega^{\beta}}), \ DH(0) \cdot \eta= \frac{1}{2} L\lp \frac{\eta}{c} \rp,
\]
where $c: [-T, T] \into \R$ is a non vanishing positive even function such that $c(\pm T)= 1/ 2$.

\end{lemma}

\begin{proof}
Let $X^{\eta_t}$ be a differentiable family in the parameter $t$ such that $\eta_0= 0$. We know that:
\[
\lr \frac{d}{dt} \rb_{t= 0} H(\eta_t)= \frac{1}{2} L \bigg{\lan} \lr \frac{d}{dt} \rb_{t= 0} X^{\eta_t}, N \bigg{\ran},
\]
where $N$ is the upward pointing unit normal to $X^0$. A computation gives:
\begin{multline*}
\bigg{\lan} \lr \frac{d}{dt} \rb_{t= 0} X^{\eta_t}, N \bigg{\ran}= \lr \frac{d\eta_t}{dt} \rb_{t= 0} \Bigg{(} \chi \frac{8r^2+ (1- \beta)(1- r^2)^2}{(1- r^2)^2} \\
+ (1- \chi) \frac{\sqrt{16\beta r^2- (1- \beta)^2(1- r^2)^2}}{4r(1+ r^2)} |h_{\beta}|(1- r^2) \Bigg{)},
\end{multline*}
which is the exact expression of $1/ c$. We can see that $c$ is a positive even function only of the variable $s$, and we have:
\[
\chi(\pm T)= 0, \quad r(\pm T)= 1 \quad \text{and} \quad \Lim{|s|}{T} \big{(} 1- r^2(s) \big{)} \big{|} h_{\beta}(s) \big{|}= \frac{4}{\sqrt{\beta}},
\]
from which we deduce $c(\pm T)= 1/ 2$.
\end{proof}

We use a similar definition to Section~\ref{sec:graph} for the compactified mean curvature operator:
\[
\ov{H}: \xi \in \Cc^{2, \alpha}(\ov{\Omega^{\beta}}) \mapsto \sqrt{|g(0)|} \lp H(2c\xi)- \frac{1}{2} \rp \in \Cc^{0, \alpha}(\ov{\Omega^{\beta}}).
\]
The compactified Jacobi operator is still $\ov{L}= D\ov{H}(0)$, $\ov{L}_0$ is its restriction to $\Cc^{2, \alpha}_0(\ov{\Omega^{\beta}})$ and $K, K^{\bot}, K_0^{\bot}$ are defined as before. The non degeneracy property of $A_{\beta}$ means $\ker \ov{L}_0= \R \phi$ with $\phi= \lan N, e_3 \ran$ the vertical component of $N$.

Again, define $\mu: (\Cc^{2, \alpha}(\S^1))^2 \into \Cc^{2, \alpha}(\ov{\Omega^{\beta}})$ to be the harmonic function on $\ov{\Omega^{\beta}}$ such that $\mu(\gamma)$ has values $\gamma$ on $\partial \Omega^{\beta}$.

\medskip

The compactified Jacobi operator satisfies a Green identity similar to Proposition~\ref{prop:greeng} for entire graphs:

\begin{proposition}[Green identity] \label{prop:greena}

For any $u, v \in \Cc^{2, \alpha}(\ov{\Omega^{\beta}})$, the compactified Jacobi operator satisfies the following identity:
\begin{multline*}
\int_{\ov{\Omega^{\beta}}} \lp u\ov{L} v- v\ov{L} u \rp d\ov{A} = \sqrt{\beta} \int_0^{2\pi} \lr \lp u\der[v]{s}- v\der[u]{s} \rp \rb_{s= T} d\theta \\
-\sqrt{\beta} \int_0^{2\pi} \lr \lp u\der[v]{s}- v\der[u]{s} \rp \rb_{s= -T} d\theta,
\end{multline*}
with $d\ov{A}$ the Lebesgue measure on $\ov{\Omega^{\beta}}$.

\end{proposition}

\nidt And we also have the equivalent of Corollary~\ref{cor:nosolg}:

\begin{corollary} \label{cor:nosola}

There is no solution $u \in \Cc^{2, \alpha}(\ov{\Omega^{\beta}})$ to the equation:
\[
\lac \begin{array}{ll}
\ov{L} u= 0 & \text{ on } \ov{\Omega^{\beta}} \\
u|_{\partial \Omega^{\beta}}= (1, -1)
\end{array} \rr.
\]

\end{corollary}

As in Section~\ref{subsec:gendef}, let $\Pi_K$ and $\Pi_{K^{\bot}}$ be the orthogonal projections on $K$ and $K^{\bot}$. Lemma~\ref{lem:isomg} still holds:

\begin{lemma}

Consider the map $\Phi: (\Cc^{2, \alpha}(\S^1))^2 \times \R \times K_0^{\bot} \into K^{\bot}$ defined by:
\[
\Phi(\gamma, \lambda, \sigma)= \Pi_{K^{\bot}} \circ \ov{H} \big{(} \mu(\gamma)+ \lambda \phi+ \sigma \big{)}.
\]
Then $D_3 \Phi(0, 0, 0): K_0^{\bot} \into K^{\bot}$ is an isomorphism.
\end{lemma}

We can apply again the implicit function theorem to $\Phi$, which states that there exist an open neighborhood $U$ of $(0, 0)$ in $(\Cc^{2, \alpha}(\S^1))^2 \times \R$ and a unique smooth map $\sigma: U \into K_0^{\bot}$ such that:
\[
\a (\gamma, \lambda) \in U, \ \Phi \big{(} \gamma, \lambda, \sigma(\gamma, \lambda) \big{)}= 0.
\]
We define similarly the smooth maps $\xi_0: U \into \Cc^{2, \alpha}(\ov{\Omega^{\beta}})$, $\eta_0: U \into \Cc^{2, \alpha}(\ov{\Omega^{\beta}})$ and $\kappa_0: U \into K$ by:
\begin{gather*}
\xi_0(\gamma, \lambda)= \mu(\gamma)+ \lambda \phi+ \sigma(\gamma, \lambda), \quad \eta_0(\gamma, \lambda)= 2c\xi_0(\gamma, \lambda) \\
\text{and} \quad \kappa_0(\gamma, \lambda)= \Pi_K \circ \ov{H} \big{(} \xi_0(\gamma, \lambda) \big{)}.
\end{gather*}
Also, if an annulus, defined on $\Omega^{\beta}$, admits $X^{\eta_0(\gamma, \lambda)}$ as a parametrization, we say that $\lac \gamma, \lambda \rac$ are the \emph{data} of the annulus with respect to $A_{\beta}$.

\medskip

Properties of $\xi_0$ and $\eta_0$ are similar to those of $\xi_a$ and $\eta_a$ in Section~\ref{subsec:gendef}:

\begin{lemma} \label{lem:etaa}

The maps $\eta_0$ and $\xi_0$ have the following properties:

\begin{enumerate}

	\item $\xi_0(0, 0)= 0$ and $\eta_0(0, 0)= 0$.

	\item $\a (\gamma, \lambda) \in U, \ \eta_0(\gamma, \lambda)|_{\partial \Omega^{\beta}}= \gamma$.

	\item $D_2  \xi_0(0, 0): \lambda \in \R \mapsto \lambda \phi \in \Cc^{2, \alpha}(\ov{\Omega^{\beta}})$.

\end{enumerate}

\end{lemma}

\begin{remark} \label{rk:deformnonrot}

With few more technical details, the method adapts to deform more general annuli that we call \emph{$\beta$-deformable}; a CMC-$1/ 2$ annulus $A$ being $\beta$-deformable when:

\begin{itemize}

	\item $A$ admits a cylindrical parametrization $X^b$ as in \eqref{eq:paramconfa}, with $b \in \Cc^{2, \alpha}(\ov{\Omega^{\beta}})$;

	\item $A$ is non degenerate;

	\item the values at infinity are the couple $(\gamma^b_+, \gamma^b_-)= b|_{\partial \Omega^{\beta}}$ satisfying the condition:
	\[
	|e^{-\gamma^b_+}|_{L^2(\S^1)}= |e^{-\gamma^b_-}|_{L^2(\S^1)},
	\]
	which expresses the conservation of the vertical flux along the annulus.

\end{itemize}

Such annuli can be found for instance among deformations of the rotational example $A_{\beta}$. Indeed, the only hypothesis which is not guaranteed by the very construction of the deformations is the non degeneracy, but since it is a local property (see Remark~\ref{rk:nondegen}) small deformations are still non degenerate.

\end{remark}

Lemma~\ref{lem:etaa} Point~2 shows that the values at infinity are still independent from the parameter $\lambda$, and the meaning of the parameter $\lambda$ is the same as in the case of entire graphs:

\begin{proposition} \label{prop:lambdaa}

Let $(\gamma, \lambda) \in U$. The surface $X^{\eta_0(\gamma, \lambda')}$ exists for any $\lambda'$ close to $\lambda$ and coincides with $X^{\eta_0(\gamma, \lambda)}$ up to a vertical translation.

\end{proposition}

The reason why now $\lambda'$ is restricted to a neighborhood of $\lambda$ is that the map $X^{\eta_0(\gamma, \lambda')}$ is an immersion only for $\eta_0(\gamma, \lambda')$ small enough in the $\Cc^{2, \alpha}$-topology. 

\medskip

We are now interested in deformations $X^{\eta_0(\gamma, \lambda)}$ of the annulus $A_{\beta}$ that are CMC-$1/ 2$, which means deformations such that $\kappa_0(\gamma, \lambda)= 0$. Consider $\Uu= \kappa_0^{-1}(\lac 0 \rac) \cap U$. Using Proposition~\ref{prop:lambdaa}, we can take $\Uu= \Gamma \times I_{\beta}$ with $\Gamma$ a connected subset of $(\Cc^{2, \alpha}(\S^1))^2$ and $I_{\beta}$ an open interval.

\begin{proposition} \label{prop:locala}

The set $\Gamma$ is a codimension $1$ smooth submanifold of $(\Cc^{2, \alpha}(\S^1))^2$ which is a subset of:
\[
\lac (\gamma_+, \gamma_-) \in (\Cc^{2, \alpha}(\S^1))^2 \lb |e^{-\gamma_+}|_{L^2(\S^1)}= |e^{-\gamma_-}|_{L^2(\S^1)} \rr \rac.
\]

\end{proposition}

\begin{proof}
As in Proposition~\ref{prop:localg}, if $\kappa_0$ is a submersion at $(0, 0)$, then it is a submersion in a neighborhood of $(0, 0)$ in $(\Cc^{2, \alpha}(\S^1))^2 \times \R$ and, up to a restriction, $\Gamma$ is a smooth submanifold of $(\Cc^{2, \alpha}(\S^1))^2$ of codimension $1$. Again $D_2 \kappa_0(0, 0)= 0$ since:
\begin{align*}
D_2 \kappa_0(0, 0) \cdot 1 & = \lr \frac{d}{dt} \rb_{t= 0} \kappa_0(0, t)= \lr \frac{d}{dt} \rb_{t= 0} \ov{H} \big{(} \xi_0(0, t) \big{)} \\
                           & = \ov{L} \big{(} D_2 \xi_0(0, 0) \cdot 1 \big{)}= \ov{L}_0(\phi)= 0,
\end{align*}
with $\phi \in K$. Consider $\gamma= (1, -1) \in (\Cc^{2, \alpha}(\S^1))^2$ and compute:
\begin{align*}
D_1 \kappa_0(0, 0) \cdot \gamma & = \lr \frac{d}{dt} \rb_{t= 0} \kappa_0(t\gamma, 0)= \lr \frac{d}{dt} \rb_{t= 0} \ov{H} \big{(} \xi_0(t\gamma, 0) \big{)} \\
                                & = \ov{L} \big{(} D_1 \xi_0(0, 0) \cdot \gamma \big{)}.
\end{align*}
Lemma~\ref{lem:etaa} Point~2 implies $\big{(} D_1 \xi_0(0, 0) \cdot \gamma \big{)}|_{\partial \D}= (1, -1)$ and using Corollary~\ref{cor:nosola}, we know that $D_1 \kappa_0(0, 0) \cdot (1, -1)$ is not identically zero.

Consider a smooth path $\gamma_t= ((\gamma_+)_t, (\gamma_-)_t)$ in $\Gamma$ with $\gamma_0= 0$ and tangent vector at $t$ $\dot{\gamma}_t= (\dot{(\gamma_+)}_t, \dot{(\gamma_-)}_t)$. Note that similarly:
\[
0= D\kappa_0(0, 0) \cdot (\dot{\gamma}_0, 0)= \lr \frac{d}{dt} \rb_{t= 0} \kappa_0(\gamma_t, 0)= \ov{L} \big{(} D_1 \xi_0(0, 0) \cdot (\dot{\gamma}_0, 0) \big{)}.
\]
Denote $v= D_1 \xi_0(0, 0) \cdot (\dot{\gamma}_0, 0) \in \ker \ov{L}$. Knowing that:
\begin{gather*}
\phi|_{s= T}= \phi|_{s= -T}= 0, \quad \lr \der[\phi]{s} \rb_{s= T}= \lr \der[\phi]{s} \rb_{s= -T}= -1, \\
v|_{s= T}= \dot{(\gamma_+)}_0 \quad \text{and} \quad v|_{s= -T}= \dot{(\gamma_-)}_0,
\end{gather*}
apply Proposition~\ref{prop:greena} to $\phi$ and $v$:
\begin{align} \label{eq:tangenta}
0 & = \int_{\ov{\Omega^{\beta}}} (\phi \ov{L} v- v\ov{L} \phi) d\ov{A}= \sqrt{\beta} \int_0^{2\pi} \dot{(\gamma_+)}_0 d\theta- \sqrt{\beta} \int_0^{2\pi} \dot{(\gamma_-)}_0 d\theta \nonumber \\
  & = 2\pi \sqrt{\beta} \lp \Big{\lan} \dot{(\gamma_+)}_0, e^{-2(\gamma_+)_0} \Big{\ran}_{L^2(\S^1)}- \Big{\lan} \dot{(\gamma_-)}_0, e^{-2(\gamma_-)_0} \Big{\ran}_{L^2(\S^1)} \rp.
\end{align}

It remains to show that equality~\eqref{eq:tangenta} is indeed true for any $t$, therefore integrating it with respect to $t$, we would obtain:
\[
|e^{-(\gamma_+)_t}|_{L^2(\S^1)}^2- |e^{-(\gamma_-)_t}|_{L^2(\S^1)}^2= |e^{-(\gamma_+)_0}|_{L^2(\S^1)}^2- |e^{-(\gamma_-)_0}|_{L^2(\S^1)}^2= 1- 1= 0.
\]
For a fixed $t$, we consider the reparametrized path $\gamma'_s= \gamma_{s+ t}$ and denote $b= \eta_0(\gamma_t, 0)$. The immersion $X^b$ satisfies the conditions of Remark~\ref{rk:deformnonrot} and can be deformed. The previous method can be applied to $X^b$ and the result~\eqref{eq:tangenta} applies to $\gamma'_s$ i.e.:
\begin{multline*}
\frac{d}{dt} \lp |e^{-(\gamma_+)_t}|_{L^2(\S^1)}^2- |e^{-(\gamma_-)_t}|_{L^2(\S^1)}^2 \rp \\
= \Big{\lan} \dot{(\gamma_+)}_t, e^{-2(\gamma_+)_t} \Big{\ran}_{L^2(\S^1)}- \Big{\lan} \dot{(\gamma_-)}_t, e^{-2(\gamma_-)_t} \Big{\ran}_{L^2(\S^1)}= 0,
\end{multline*}
for any $t$ as desired.
\end{proof}

The condition on the values at infinity defining $\Gamma$ is indeed the conservation of the vertical flux in the deformed annuli. Moreover, with the structure of $\Gamma$ given in Proposition~\ref{prop:locala}, the proof of Theorem~\ref{thm:structg} adapts to $\beta$-deformable annuli, giving the following result:

\begin{theorem}

The family of $\beta$-deformable annuli can be endowed with a structure of infinite dimensional smooth manifold.

\end{theorem}

\subsection{Annuli with non aligned ends}

For minimal surfaces in $\R^3$, one can define two Nœther vector-invariants associated to isometries, namely the flux ---~associated to translations~--- and the torque ---~associated to rotations. In the case of a minimal catenoidal end with growth $\alpha> 0$ and vertical axis $\lac x_1= u, x_2= v \rac$, the flux and the torque are respectively $(0, 0, 2\pi \alpha)$ and $2\pi \alpha (v, -u, 0)$. In other words, the growth and the position of the axis of the end are determined by the vertical component of the flux and horizontal components of the torque.

In $\H^2 \times \R$, Nœther invariants are constructed similarly but the torque is not a vector anymore, since remain only rotations around vertical axis. In the case of a vertical rotational end with parameter $\beta> 0$, the flux is vertical with third component $\beta$ and the torque is always zero, no matter where the rotation axis is situated. The fact the position of the axis is no longer caught by Nœther invariants, indicates that the construction of CMC-$1/ 2$ annuli with vertical ends should be more flexible regarding the relative positions of the axis of the ends.

\begin{theorem} \label{thm:nonalign}

There exist a CMC-$1/ 2$ annulus in $\H^2 \times \R$ with vertical ends such that each end of the annulus is asymptotic ---~regarding the horizontal hyperbolic distance~--- to a rotational example and the (vertical) axes of the rotational asymptotes are different.

\end{theorem}

\begin{proof}
Fix $\beta> 0$, $\beta \neq 1$. From Proposition~\ref{prop:invisom}, we know that, in the Poincar{\'e} disk model~\eqref{eq:modpoinc}, a horizontal translation of $w_{\epsilon}= \epsilon e^{i\theta_0} \in \D^*$ changes the top value at infinity of the rotational annulus $A_{\beta}$ into:
\[
\gamma_{\epsilon}(\theta)= \log \lp \frac{|1- \epsilon e^{i(\theta- \theta_0)}|}{\sqrt{1- \epsilon^2}} \rp.
\]
A direct computation shows:
\[
|e^{-\gamma_{\epsilon}}|_{L^2(\S^1)}= 1 \quad \text{and} \quad |\gamma_{\epsilon}|_{\Cc^{2, \alpha}(\S^1)} \leq C\epsilon \quad \text{with} \quad C \in \R_+^*.
\]
Thus, for $\epsilon$ sufficiently small, we have $\big{(} (\gamma_{\epsilon}, 0), 0 \big{)} \in \Uu_0$ and the CMC-$1/ 2$ annulus $X^{\eta_0 \big{(} (\gamma_{\epsilon}, 0), 0 \big{)}}$ exists.

Moreover, the top end of $X^{\eta_0 \big{(} (\gamma_{\epsilon}, 0), 0 \big{)}}$ is asymptotic to the top end of the image of $S_0$ under the horizontal translation by $w_{\epsilon}$ ---~since it has the same value at infinity~--- and is therefore asymptotically rotational. Similarly, the bottom end of $X^{\eta_0 \big{(} (\gamma_{\epsilon}, 0), 0 \big{)}}$ is asymptotically rotational, being asymptotic to the bottom end of $S_0$.

And finally, the ends of $X^{\eta_0 \big{(} (\gamma_{\epsilon}, 0), 0 \big{)}}$ are not aligned since the axis of the top end is $\lac w_{\epsilon} \rac \times \R$ and the one of the bottom end is $\lac 0 \rac \times \R$.
\end{proof}

\begin{remark}

In the proof of Theorem~\ref{thm:nonalign}, we see that the ends of the constructed annulus are asymptotic to \emph{the same} rotational example, up to isometries. This is indeed a necessary condition since the ends have to preserve the vertical flux, which is determined by the parameter $\beta$ of the rotational annulus ---~namely, the vertical flux of the annulus $A_{\beta}$ is $2\pi(1- \beta)$.

\end{remark}

\appendix

\section{Compactification of the mean curvature} \label{sec:comp}

Consider the product metric $\sigma+ dx_3^2$ on $\D \times \R$ where:
\[
\sigma= F^* ds_P^2 \quad \text{and} \quad F: (r, \theta) \in \D \mapsto \frac{2r}{1+ r^2} e^{i\theta} \in \H^2,
\]
in the Poincar{\'e} disk model~\eqref{eq:modpoinc}. To ease the notations, we use indices $1, 2$ for quantities respectively related to coordinates $r, \theta$ on $\D$. In matrix terms, the metric is $\sigma= (\sigma_{ij})$ with:
\begin{gather*}
\sigma_{11}= \frac{16}{(1- r^2)^2}, \quad \sigma_{12}= \sigma_{21}= 0, \quad \sigma_{22}= \frac{16r^2(1+ r^2)^2}{(1- r^2)^4} \\
\text{and} \quad |\sigma|= \lp  \frac{16r(1+ r^2)}{(1- r^2)^3} \rp^2.
\end{gather*}
The Christoffel symbols $(\Gamma_{ij}^k)$ associated to $\sigma$ for the Levi-Civita connection verify:
\[
\Gamma_{ij}^k= \frac{1}{2} \sum_m \sigma^{km} \lp \partial_i \sigma_{jm}+ \partial_j \sigma_{im}- \partial_m \sigma_{ij} \rp,
\]
which means:
\begin{gather*}
\Gamma_{11}^1= \frac{2r}{1- r^2}, \quad \Gamma_{12}^2= \Gamma_{21}^2= \frac{1+ 6r^2+ r^4}{r(1+ r^2)(1- r^2)} \\
\text{and} \quad \Gamma_{22}^1= -\frac{r(1+ r^2)(1+ 6r^2+ r^4)}{(1- r^2)^3},
\end{gather*}
the other terms being zero.

Fix $\Omega \in \Dd$. A surface in $S \in \Ee$ defined on $\Omega$ with graph coordinates at infinity:
\[
(r, \theta) \in \Omega \mapsto \big{(} F(r, \theta), h(\eta) \big{)} \quad \text{with} \quad \eta \in \Cc^{2, \alpha}(\ov{\Omega}) \quad \text{and} \quad h(\eta)= 2e^{\eta} \frac{1+ r^2}{1- r^2},
\]
can be reparametrized as the actual graph of the function $h(\eta): \Omega \into \R$ in $\D \times \R$ endowed with metric $\sigma+ dx_3^2$. As shown by J. Spruck \cite{Sp}, the metric $g(\eta)= (g_{ij}(\eta))$ induced by $h(\eta)$ is given by:
\[
g_{ij}(\eta)= \sigma_{ij}+ \partial_i h(\eta) \partial_j h(\eta),
\]
and denoting $\eta_i= \partial_i \eta$, for $i= 1, 2$, we obtain:
\begin{gather*}
g_{11}(\eta)= \frac{16(1+ r^2)^2 e^{2\eta}}{(1- r^2)^4} \Bigg{[} 1+ \frac{2r\eta_1}{1+ r^2} (1- r^2)+ \lp \frac{\eta_1^2}{4}+ \frac{e^{-2\eta}- 1}{(1+ r^2)^2} \rp (1- r^2)^2 \Bigg{]}, \\
g_{12}(\eta)= \frac{8(1+ r^2)^2 \eta_2 e^{2\eta}}{(1- r^2)^3} \lc \frac{2r}{1+ r^2}+ \frac{\eta_1}{2} (1- r^2) \rc \\
\text{and} \quad g_{22}(\eta)= \frac{16r^2 (1+ r^2)^2}{(1- r^2)^4} \lc 1+ \frac{\eta_2^2 e^{2\eta}}{4r^2} (1- r^2)^2 \rc.
\end{gather*}
The determinant $|g(\eta)|$ of the induced metric is:
\[
|g(\eta)|= \lp \frac{16r(1+ r^2)^2 e^{\eta}}{(1- r^2)^4} \rp^2 w^2(\eta)
\]
with $w(\eta)$ denoting:
\begin{multline*}
w(\eta)= \Bigg{[} 1+ \frac{2r\eta_1}{1+ r^2} (1- r^2)+ \lp \frac{\eta_1^2}{4}+ \frac{e^{-2\eta}- 1}{(1+ r^2)^2} \rp (1- r^2)^2 \\
+\frac{\eta_2^2}{4r^2 (1+ r^2)^2}(1- r^2)^4 \Bigg{]}^{1/ 2}.
\end{multline*}
In the metric $\sigma+ dx_3^2$, the mean curvature $H(\eta)$ of $S$ can be expressed as:
\begin{gather*}
H(\eta)= \frac{1}{2} \Div_{\sigma} \lp \frac{\del_{\sigma} h(\eta)}{W(\eta)} \rp= \frac{1}{2W(\eta)} \sum_{i, j} g^{ij}(\eta) \lp \partial_{ij} h(\eta)- \sum_k \Gamma_{ij}^k \partial_k h(\eta) \rp \\
\text{with} \quad W(\eta)= \sqrt{1+ |\del_{\sigma} h(\eta)|_{\sigma}^2}= \frac{(1+r^2)e^{\eta}}{1- r^2} w(\eta),
\end{gather*}
where the quantities are computed with respect to the metric $\sigma$ on $\D$, and $g^{-1}(\eta)= (g^{ij}(\eta))$.

In order to ease the notations, denote:
\[
H_{ij}(\eta)= g^{ij}(\eta) \lp \partial_{ij} h(\eta)- \sum_k \Gamma_{ij}^k \partial_k h(\eta) \rp.
\]
For $H_{11}(\eta)$, compute:
\begin{align*}
H_{11}(\eta) & = g^{11}(\eta) \lp \partial_{11} h(\eta)- \Gamma_{11}^1 \partial_1 h(\eta) \rp \\
             & = \frac{W(\eta)}{w^3(\eta)} e^{-2\eta} (1- r^2)^2 \Bigg{[} \frac{1}{2(1+ r^2)^2}+ \frac{r\eta_1}{2(1+ r^2)^3} (1- r^2) \\
             & \hspace{1em}+ R_{11} (1- r^2)^2 \Bigg{]}+ 2W(\eta) \frac{g^{11}(\eta)}{w(\eta)} \eta_{11},
\end{align*}
with $R_{11}= R_{11}(r, \eta, D\eta)$ defined on $\Omega \cup \partial \D$, identically zero if $\eta= 0$ and real-analytic in its variables. For $H_{12}(\eta)$:
\begin{align*}
H_{12}(\eta) & = g^{12}(\eta) \lp \partial_{12} h(\eta)- \Gamma_{12}^2 \partial_2 h(\eta) \rp \\
             & = \frac{W(\eta)}{w^3(\eta)} R_{12} (1- r^2)^4+ 2W(\eta) \frac{g^{12}(\eta)}{w(\eta)} \eta_{12},
\end{align*}
again with $R_{12}= R_{12}(r, \eta, D\eta)$ defined on $\Omega \cup \partial \D$, zero if $\eta= 0$ and real-analytic in its variables. And for $H_{22}(\eta)$:
\begin{align*}
H_{22}(\eta) & = g^{22}(\eta) \lp \partial_{22} h(\eta)- \Gamma_{22}^1 \partial_1 h(\eta) \rp \\
             & = \frac{W(\eta)}{w^3(\eta)} \Bigg{[} \frac{1+ 6r^2+ r^4}{2(1+ r^2)^2}+ \frac{(5- 10r^2+ 29r^4)\eta_1}{2r(1+ r^2)^3} (1- r^2) \\
             & \hspace{1em}+ \frac{4r^2}{(1+ r^2)^2} \lp \frac{3\eta_1^2}{4}+ \frac{e^{-2\eta}- 1}{(1+ r^2)^2} \rp (1- r^2)^2+ \frac{\eta_1}{r(1+ r^2)} \Bigg{(} \frac{\eta_1^2}{4} \\
             & \hspace{1em}+ \frac{e^{-2\eta}- 1}{(1+ r^2)^2} \Bigg{)} (1- r^2)^3+ R_{22} (1- r^2)^4 \Bigg{]}+ 2W(\eta) \frac{g^{22}(\eta)}{w(\eta)} \eta_{22},
\end{align*}
with $R_{22}= R_{22}(r, \eta, D\eta)$ defined on $\Omega \cup \partial \D$, zero if $\eta= 0$ and real-analytic in its variables. Hence, a Taylor expansion of the mean curvature $H(\eta)$ is:
\begin{align*}
H(\eta) & = \frac{1}{w(\eta)} \lp g^{11}(\eta) \eta_{11}+ 2g^{12}(\eta) \eta_{12}+ g^{22}(\eta) \eta_{22} \rp \\
        & \hspace{1em}+ \frac{1}{2w^3(\eta)} \Bigg{[} 1+ \frac{3r\eta_1}{1+ r^2} (1- r^2)+ \frac{6r^2}{(1+ r^2)^2} \Bigg{(} \frac{\eta_1^2}{2} \\
        & \hspace{1em}+ \frac{e^{-2\eta}- 1}{(1+ r^2)^2} \Bigg{)} (1- r^2)^2+ \frac{\eta_1}{2r(1+ r^2)} \bigg{(} \frac{\eta_1^2}{2} \\
        & \hspace{1em}+ \frac{3(e^{-2\eta}- 1)}{(1+ r^2)^2} \Bigg{)} (1- r^2)^3 \Bigg{]}+ R_H (1- r^2)^4,
\end{align*}
with as before $R_H= R_H(r, \eta, D\eta)$ defined on $\Omega \cup \partial \D$, identically zero if $\eta= 0$ and real-analytic in its variables.

The Taylor expansion of $w^{-3}(\eta)$ is the following:
\begin{align*}
\frac{1}{w^3(\eta)} & = 1- \frac{3r\eta_1}{1+ r^2} (1- r^2)- \frac{3}{2(1+ r^2)} \lp 4r^2 \eta_1^2+ (e^{-2\eta}- 1) \rp (1- r^2)^2 \\
                    & \hspace{1em}- \frac{5r\eta_1}{(1+ r^2)^3} \lp 2r^2 \eta_1^2- \frac{3(e^{-2\eta}- 1)}{2} \rp (1- r^2)^3+ R_w (1- r^2)^4,
\end{align*}
with $R_w= R_w(r, \eta, D\eta)$ defined on $\Omega \cup \partial \D$, zero if $\eta= 0$ and real-analytic in its variables. Finally, we obtain:
\begin{equation} \label{eq:devh}
H(\eta)= \frac{1}{2}+ \frac{1}{w(\eta)} \lp g^{11}(\eta) \eta_{11}+ 2g^{12}(\eta) \eta_{12}+ g^{22}(\eta) \eta_{22} \rp+ R(1- r^2)^4,
\end{equation}
with $R= R(r, \eta, D\eta)$ defined on $\Omega \cup \partial \D$, identically zero if $\eta= 0$ and real-analytic in its variables.

Taking $\eta= a+ \xi$ with $a, \xi \in \Cc^{2, \alpha}(\ov{\Omega})$, the Taylor expansion \eqref{eq:devh} reads:
\begin{gather*}
H(a+ \xi)= H(a)+ \frac{1}{\sqrt{|g(a)|}} \sum_{i, j} A_{ij} \xi_{ij}+ \frac{1}{\sqrt{|g(a)|}} B, \\
\begin{array}{lr@{~=~}l}
\text{with} & A_{11} & \dst \frac{1}{w(a+ \xi)} \sqrt{|g(a)|} g^{11}(a+ \xi)= \frac{1}{w(a+ \xi)} \frac{g_{22}(a+ \xi)}{\sqrt{|g(a)|}} \espf \\
            &        & e^{-a}+ O(1- r^2), \espf \\
            & A_{12} & \dst \frac{1}{w(a+ \xi)} \sqrt{|g(a)|} g^{12}(a+ \xi)= -\frac{1}{w(a+ \xi)} \frac{g_{12}(a+ \xi)}{\sqrt{|g(a)|}} \espf \\
            &        & O(1- r^2) \espf \\
\text{and}  & A_{22} & \dst \frac{1}{w(a+ \xi)} \sqrt{|g(a)|} g^{22}(a+ \xi)= \frac{1}{w(a+ \xi)} \frac{g_{11}(a+ \xi)}{\sqrt{|g(a)|}} \espf \\
            &        & e^a+ O(1- r^2).
\end{array}
\end{gather*}
Moreover $A_{ij}= A_{ij}(r, a, \xi, D\xi)$ and $B= B(r, a, \xi, D\xi)$ are defined on $\Omega \cup \partial \D$ and real-analytic in their variables, the matrix $A= (A_{ij})$ is coercive on $\Omega \cup \partial \D$, and $B$ is identically zero if $\xi= 0$.

\newpage

\bibliographystyle{amsplain}
\providecommand{\bysame}{\leavevmode\hbox to3em{\hrulefill}\thinspace}
\providecommand{\MR}{\relax\ifhmode\unskip\space\fi MR }
% \MRhref is called by the amsart/book/proc definition of \MR.
\providecommand{\MRhref}[2]{%
  \href{http://www.ams.org/mathscinet-getitem?mr=#1}{#2}
}
\providecommand{\href}[2]{#2}

% \begin{comment}
\vfill

\nidt S{\'e}bastien \textsc{Cartier}, Universit{\'e} Paris-Est, LAMA (UMR 8050), UPEMLV, UPEC, CNRS, F-94010, Cr{\'e}teil, France \\
e-mail: \verb+sebastien.cartier@u-pec.fr+

\medskip

\nidt Laurent \textsc{Hauswirth}, Universit{\'e} Paris-Est, LAMA (UMR 8050), UPEMLV, UPEC, CNRS, F-77454, Marne-la-Vall{\'e}e, France \\
e-mail: \verb+hauswirth@univ-mlv.fr+
% \end{comment}

\end{document}